\documentclass[11pt,reqno]{amsart}
\usepackage[top=1.25in,bottom=1.25in,left=1in,right=1in]{geometry}
\usepackage{amsmath,amssymb,amsthm,pinlabel,enumitem}
\usepackage{tikz}
\usetikzlibrary{shapes.geometric}
\usepackage[final, 
colorlinks=true,
pdftitle={Chain flaring and $L^2$--torsion of free-by-cyclic groups},
pdfauthor={Matthew Clay}]{hyperref}
\usepackage{ifthen}

\theoremstyle{plain}
\newtheorem{theorem}{Theorem}
\newtheorem{lemma}[theorem]{Lemma}
\newtheorem{proposition}[theorem]{Proposition}
\newtheorem{corollary}[theorem]{Corollary}

\newtheorem*{claim*}{Claim}

\theoremstyle{definition}
\newtheorem{definition}[theorem]{Definition}

\newtheorem{example}[theorem]{Example}
\newtheorem{remark}[theorem]{Remark}

\numberwithin{theorem}{section}
\numberwithin{equation}{section}

\newcommand{\fakeenv}{} 

\newenvironment{restate}[2]  
{ 
 \renewcommand{\fakeenv}{#2} 
 \theoremstyle{plain} 
 \newtheorem*{\fakeenv}{#1~\ref{#2}} 
 \begin{\fakeenv}
}
{
 \end{\fakeenv}
}

\newcommand{\CC}{\mathbb{C}} 
 
\newcommand{\FF}{\mathbb{F}}
\newcommand{\NN}{\mathbb{N}} 
\newcommand{\RR}{\mathbb{R}}

\newcommand{\QQ}{\mathbb{Q}}
\newcommand{\ZZ}{\mathbb{Z}}

\newcommand{\calC}{\mathcal{C}}

\newcommand{\calH}{\mathcal{H}}

\newcommand{\calN}{\mathcal{N}}

\newcommand{\frakE}{\mathfrak{E}}
\newcommand{\frakF}{\mathfrak{F}}
\newcommand{\frakK}{\mathfrak{K}}

\newcommand{\sfo}{\mathsf{o}}
\newcommand{\sft}{\mathsf{t}}

\newcommand{\sfE}{\mathsf{E}}

\newcommand{\sfV}{\mathsf{V}}

\newcommand{\bp}{\mathbf{p}}  
\newcommand{\bq}{\mathbf{q}}

\newcommand{\hGamma}{\Gamma'}

\newcommand{\tgamma}{{\widetilde \gamma}}

\newcommand{\tGamma}{{\widetilde \Gamma}}

\newcommand{\tf}{{\tilde f}}

\newcommand{\tast}{{\tilde *}}

\newcommand{\tH}{{\widetilde H}}

\newcommand{\tX}{{\widetilde X}}

\newcommand{\abs}[1]{\left\lvert {#1} \right\rvert}

\newcommand{\norm}[1]{\left\| {#1} \right\|} 
 
\newcommand{\I}[1]{\langle #1 \rangle}

\newcommand{\bd}{\partial}
\newcommand{\from}{\colon\thinspace}

\newcommand{\inv}{^{-1}}

\newcommand{\detG}[1][1]{%
\ifthenelse{\equal{#1}{1}}%
{\operatorname{det}_{G}}%
{\operatorname{det}_{#1}}%
}

\newcommand{\param}%
	{{\mathchoice{\mkern1mu\mbox{\raise2.2pt\hbox{$\centerdot$}}\mkern1mu}%
	{\mkern1mu\mbox{\raise2.2pt\hbox{$\centerdot$}}\mkern1mu}%
	{\mkern1.5mu\centerdot\mkern1.5mu}{\mkern1.5mu\centerdot\mkern1.5mu}}}

\newcommand{\PF}{\mathcal{EG}}

\DeclareMathOperator{\Aut}{Aut}

\DeclareMathOperator{\Mat}{Mat}

\DeclareMathOperator{\Out}{Out}
\DeclareMathOperator{\clos}{clos}
\DeclareMathOperator{\id}{id}
\DeclareMathOperator{\img}{im}

\DeclareMathOperator{\sign}{sgn}
\DeclareMathOperator{\supp}{supp}
\DeclareMathOperator{\stab}{stab}
\DeclareMathOperator{\Tr}{tr}

\begin{document}


\title[Chain flaring and $L^{2}$--torsion of free-by-cyclic groups]{Chain flaring and $L^{2}$--torsion of free-by-cyclic groups}

\author[M.~Clay]{Matt Clay}
\address{Dept.\ of Mathematics \\
University of Arkansas\\
Fayetteville, AR 72701}
\email{\href{mailto:mattclay@uark.edu}{mattclay@uark.edu}}

\begin{abstract}
We introduce a condition on the monodromy of a free-by-cyclic group, $G_\phi$, called the \emph{chain flare condition}, that implies that the $L^2$--torsion, $\rho^{(2)}(G_\phi)$, is non-zero.  We conjecture that this condition holds whenever the monodromy is exponentially growing.
\end{abstract}

\maketitle


\section{Introduction}

The $L^2$--torsion, denoted $\rho^{(2)}(G)$, is an analytical group invariant that is well-defined for a large class of $L^2$--acyclic groups, i.e., a group whose $L^2$--homology vanishes.  (For the remainder, when we speak of the $L^2$--torsion of an $L^2$--acyclic group $G$, we implicitly assume that $G$ is in this class, which conjecturally includes all $L^2$--acyclic groups; see~\cite[Section~13]{bk:Luck02} and \cite[Section~13]{ar:Luck16}.) This invariant is a real number and it behaves similarly to Euler characteristic in the sense that it is multiplicative under covers and that there exists a sum formula along pushouts.  If $G$ is $L^2$--acyclic and $G$ contains an elementary amenable normal subgroup, then it was shown by Wegner that $\rho^{(2)}(G)  = 0$~\cite{col:Wegner00}, see also~\cite[Theorem~3.113]{bk:Luck02}.  In addition, there are a variety of conjectures and partial results relating the $L^2$--torsion of a group to the growth of torsion in homology; see the survey articles by L\"uck~\cite[Section~7]{ar:Luck16} and \cite[Section~3.6]{un:Luck21}.          

In the setting of the fundamental groups of 3--manifolds, the $L^2$--torsion was computed by L\"uck--Schick~\cite{ar:LS99}.  We recall their results in the context of an orientable 3--manifold that fibers over $S^1$ as this parallels the setting considered in this paper.  Assuming for simplicity that the fiber is connected, such a 3--manifold is homeomorphic to a mapping torus:
\begin{equation*}
M_f = \raisebox{5pt}{$\Sigma \times [0,1]$} \Big/ \raisebox{-5pt}{$(x,0) \sim (f(x),1)$},
\end{equation*}
where $f \from \Sigma \to \Sigma$ is a homeomorphism of an orientable connected surface $\Sigma$.  Given such a homeomorphism $f \from \Sigma \to \Sigma$, let $\calC$ be the canonical cut system for $f$ 
and assume for simplicity that each curve in $\calC$ is fixed up to homotopy by $f$.  For each component $\Sigma_s \subseteq \Sigma - \calC$, $s = 1,\ldots, S$, we have $f(\Sigma_s) = \Sigma_s$ and the restriction of $f$ to $\Sigma_s$ determines a sub-mapping torus $M_{f,s} \subseteq M_f$.  As $\calC$ is the canonical cut system for $f$, the restriction of $f$ to $\Sigma_s$ is, up to homotopy, either periodic or pseudo-Anosov.  When the restriction of $f$ to $\Sigma_s$ is pseudo-Anosov, Thurston proved that the manifold $M_{f,_s}$ admits a complete hyperbolic metric~\cite{ar:Thurston82}.  The work of L\"uck--Schick in this setting shows that $-\rho^{(2)}(\pi_1(M_f))$ equals $\frac{1}{6\pi}$ times the sum---over the indices $1 \leq s \leq S$ where the restriction of $f$ to $\Sigma_s$ is pseudo-Anosov---of the volumes of these hyperbolic sub-mapping tori $M_{f,s}$.  In particular, the $L^2$--torsion is determined by the exponential dynamics of $f$.  Combining this with work of Gromov~\cite{ar:Gromov82}, Soma~\cite{ar:Soma81} and Thurston~\cite{un:Thurston78}, this implies that the $L^2$--torsion $-\rho^{(2)}(\pi_1(M_f))$ is proportional to the simplicial volume $\norm{M_f}$.  Similarly, combining this with work of Pieroni~\cite{un:Pieroni} this also implies that the $L^2$--torsion is proportional to the cube of the minimal volume entropy $\omega(M_f)^3$.

Building onto the established research by Algom-Kfir--Hironaka--Rafi~\cite{ar:A-KHR15}, Dowdall--Kapovich--Leininger~\cite{ar:DKL15,ar:DKL17,ar:DKL17-2}, Funke--Kielak~\cite{ar:FK18} and others of studying free-by-cyclic groups analogously to 3--manifolds that fiber over $S^1$, the aim of this paper is to study the $L^2$--torsion of a free-by-cyclic group, in particular, trying to understand when this invariant is non-zero.  A \emph{free-by-cyclic group} is a group that fits into a short exact sequence:
\begin{equation*}
1 \to \FF \to G \to \ZZ \to 1
\end{equation*}
where $\FF$ is a finitely generated free group and hence it admits a presentation as a semi-direct product:
\begin{equation*}
G = \FF\rtimes_\Phi \I{t} = \I{\FF,t \mid t\inv x t = \Phi(x) \mbox{ for } x \in \FF }
\end{equation*}
where $\Phi \in \Aut(\FF)$.  Changing the automorphism $\Phi$ within its outer automorphism  class amounts to replacing the generator $t$ by $tx$ for some $x \in \FF$ and so we are justified in denoting the above defined group by $G_{\phi}$ where $\phi = [\Phi] \in \Out(\FF)$.

Previously, building off of work by L\"uck~\cite[Section~7.4]{bk:Luck02}, the author showed how to compute $-\rho^{(2)}(G_\phi)$ using a topological representative $f \from \Gamma \to \Gamma$ of $\phi \in \Out(\FF)$ \cite{ar:Clay17-2}.  Similar to the setting of 3--manifolds that fiber over $S^1$ mentioned above, it was shown that $-\rho^{(2)}(G_\phi)$ can be expressed using a topological representative $f \from \Gamma \to \Gamma$ for $\phi \in \Out(\FF)$.  In this context, $\Gamma$ is a graph, $f$ is a homotopy equivalence and there is filtration $\emptyset = \Gamma_{0} \subset \Gamma_{1} \subset \cdots \subset \Gamma_S = \Gamma$ by subgraphs that is respected by $f$ in the sense that $f(\Gamma_s) \subseteq \Gamma_s$ for each $s = 1,\ldots, S$.  For each $1 \leq s \leq S$, there is a non-negative integer matrix $M(f)_s$ that records the number of times the image of an edge in $\Gamma_{s} - \Gamma_{s-1}$ crosses an edge in $\Gamma_{s} - \Gamma_{s-1}$.  Enlarging the filtration if necessary, we can assume that each matrix $M(f)_s$ is either the zero matrix or it is irreducible.  With this set-up $-\rho^{(2)}(G_\phi)$ is expressed as a sum over the the indices $1 \leq s \leq S$  where $M(f)_s$ is irreducible and has Perron--Frobenius eigenvalue strictly greater than 1; this subset of indices is denoted $\PF(f)$.  Each term in the summation is the logarithm of the Fuglede--Kadison determinant of an operator associated to the restriction of $f$ to the subgraph $\Gamma_s$.  Hence, as mentioned previously in the setting of 3--manifolds that fiber over $S^1$, the $L^2$--torsion of a free-by-cyclic group is determined by the exponential dynamics of the monodromy $f$.  While the simplicial volume of a free-by-cyclic group is not well-defined, the minimal volume entropy is.  It was recently shown by Bregman and the author that the $L^2$--torsion $-\rho^{(2)}(G_\phi)$ is not proportional to the square of the minimal volume entropy $\omega(G_\phi)^2$ in general~\cite{un:BC}.

L\"uck has shown that $-\rho^{(2)}(G_\phi)$ is non-negative for any free-by-cyclic group~\cite[Theorem~7.29]{bk:Luck02}.  The main result in the author's previous work \cite{ar:Clay17-2} provides an upper bound on $-\rho^{(2)}(G_\phi)$ in terms of the data previously described, namely the matrices $M(f)_s$.   In particular, it was shown that $-\rho^{(2)}(G_\phi) = 0$ when $\phi$ is polynomially growing, i.e., $\PF(f) = \emptyset$.  In this paper, we suggest a strategy to show that $-\rho^{(2)}(G_\phi) > 0$ whenever $\phi$ is exponentially growing, i.e., $\PF(f) \neq \emptyset$.  To this end, we introduce a condition, called the chain flare condition, and prve that this implies $-\rho^{(2)}(G_\phi) > 0$.    

\begin{theorem}\label{th:mahler}
Suppose that $f \from \Gamma \to \Gamma$ is a homotopy equivalence that respects the reduced filtration $\emptyset = \Gamma_0 \subset \Gamma_1 \subset \cdots \subset \Gamma_S = \Gamma$ and that $f \from \Gamma \to \Gamma$ represents the outer automorphism $\phi \in \Out(\FF)$.  If the restriction of $f$ to $\hGamma_s$ satisfies the chain flare condition relative to $\hGamma_s \cap \Gamma_{s-1}$ for each $s \in \PF(f)$, then
\begin{equation}\label{eq:mahler}
-\rho^{(2)}(G_\phi) = \sum_{s \in \PF(f)} \int_{1 < \abs{z}} \log \abs{z} d\mu_{L_{f,s}}.
\end{equation}
Moreover, each integral in~\eqref{eq:mahler} is positive and hence $-\rho^{(2)}(G_\phi) > 0$.
\end{theorem}

In Theorem~\ref{th:mahler}, the filtration $\emptyset = \Gamma_0 \subset \Gamma_1 \subset \cdots \subset \Gamma_S = \Gamma$ being reduced means that there is a single component $\Gamma'_s \subseteq \Gamma_s$ that is not contained in $\Gamma_{s-1}$.  The measures $\mu_{L_{f,s}}$ appearing in \eqref{eq:mahler} are the Brown measures~\cite{col:Brown86} associated to the operators $L_{f,s} \from L^2(G_\phi)^{n_s} \to L^2(G_\phi)^{n_s}$ ($n_s$ is the number of edges in $\Gamma_s - \Gamma_{s-1}$).  These operators are described fully in Section~\ref{subsec:compute}, but briefly, they are induced from the vertical flow in the universal cover of the mapping torus for $f$.

The chain flare condition is the linear analog of the annuli flare condition of Bestvina--Feighn~\cite{ar:BF92} that has been successfully employed by Bestvina--Feighn--Handel~\cite{ar:BFH97} and Brinkmann~\cite{ar:Brinkmann00} to prove hyperbolicity of certain free-by-cyclic groups, and more generally by Kapovich~\cite{ar:Kapovich00} and Mutanguha~\cite{ar:Mutanguha20} to prove hyperbolicity of certain ascending HNN-extensions over free groups.  Full details regarding the chain flare condition will be given in Section~\ref{sec:cfh}, but we provide a quick explanation here.  We can lift the topological representative $f \from \Gamma \to \Gamma$  of $\phi \in \Out(\FF)$ to a cellular map $\tf \from \tGamma \to \tGamma$ where $\tGamma$ is the universal cover of the graph $\Gamma$.  This map is not $\FF$--equivariant, but satisfies $\tf(gz) = \Phi_{f}\tf(z)$ where $g \in \FF$, $z \in \tGamma$ and $\Phi_{f} \in \Aut(\FF)$ represents $\phi$.  The map $\tf$ induces a map on the level of cellular 1--chains of $\tGamma$, which we denote by $A_f \from C_1(\tGamma;\QQ) \to C_1(\tGamma;\QQ)$.  In spirit, the chain flare condition asserts the existence of a constant $\lambda > 1$ such that for any 1--chain $x \in C_1(\tGamma;\QQ)$ we have:
\begin{equation*}
\lambda{\norm{A_f(x)}} \leq \max\left\{ \norm{A_f^2(x)}, \norm{x}\right\}
\end{equation*}  
---where $\norm{\param}$ is the usual $L^2$--norm---unless the 1--chain $x$ has an obvious reason why it should not satisfy this inequality, e.g., $x$ is fixed by $A_f$.  The previously mentioned work of Bestvina--Feighn--Handel~\cite{ar:BFH97} and Brinkmann~\cite{ar:Brinkmann00} implies that the chain flare condition always holds when the support of the boundary of $x$ consists of two points, i.e., $x$ is the 1--chain determined by an edge-path in $\tGamma$.

Before we give an outline of the proof of Theorem~\ref{th:mahler} and the rest of the paper, we mention some related results.  For a free-by-cyclic group $\FF \rtimes_{\Phi_f} \I{t}$, the $L^2$--torsion is the logarithm of the Fuglede--Kadison determinant of the operator given by right multiplication by $I - tJ_1(f)$ where $J_1(f)$ is a matrix with entries in the group ring $\ZZ[\FF]$, see Section~\ref{subsec:compute}.  Deninger has computed the Fuglede--Kadison determinant of similar operators for the discrete Heisenberg group~\cite{ar:Deninger11}.  Specifically, the discrete Heisenberg groups can be expressed as a semi-direct product $\ZZ^2 \rtimes_{\Phi} \I{t}$ where $t$ acts on $\ZZ^2$ via the matrix $\left[\begin{smallmatrix}1 & 1 \\ 0 & 1 \end{smallmatrix}\right]$.  Deninger studies operators given by right multiplication by $1 - xt$ where $x \in \CC[\ZZ^2]$ and expresses the logarithm of the Fuglede--Kadison determinant as an integral of $\log\abs{x}$, treating $x$ as a polynomial in two variables~\cite[Theorem~11]{ar:Deninger11}.  Funke--Kielak study a different variant of the $L^2$--torsion of free-by-cyclic groups, called the $L^2$--torsion polytope \cite{ar:FK18}, see also the later work by Kielak~\cite{ar:Kielak19}.  Together, these works show that the BNS invariant of the free-by-cyclic group is determined by the $L^2$--torsion polytope.  The connection between their work and the present work is the universal $L^2$--torsion defined by Friedl--L\"uck~\cite{ar:FL17}.  This is a certain element in the weak Whitehead group $\rho^{(2)}_u(G) \in {\rm Wh}^w (G)$ associated to the $L^2$--acyclic group $G$.  The logarithm of the Fuglede--Kadison determinant gives a homomorphism ${\rm Wh}^w(G) \to \RR$; the $L^2$--torsion $\rho^{(2)}(G)$ is the image of $\rho^{(2)}_u(G)$.  There is another homomorphism defined on ${\rm Wh}^w(G)$ by Friedl--L\"uck~\cite{ar:FL17} whose image is a polytope; the $L^2$--torsion polytope is the image of $\rho^{(2)}_u(G)$.


\subsection{Outline of the proof}\label{subsec:outline}

The proof of Theorem~\ref{th:mahler} builds off of the author's previous work~\cite{ar:Clay17-2}.  To simplify the exposition in this introductory section, we will assume that the filtration consists of a single stratum, i.e., that $S = 1$.  For notational simplicity, we will denote the operator $L_{f,1}$ simply by $L$ and $n_1$ simply by $n$ (the number of edges in $\Gamma$).  Using the property that the $L^2$--torsion is multiplicative under covers and properties of the Brown measure, the proof of Theorem~\ref{th:mahler} starts by showing that for any $k \geq 1$, we can express the $L^2$--torsion as a certain integral over $\CC$:
\begin{equation*}
-\rho^{(2)}(G_\phi) = \frac{1}{k} \int_\CC \log \abs{1-z^k} \, d\mu_{L}.
\end{equation*}    
We would like to take the limit as $k \to \infty$.  Notice that $\log\abs{1-z^k}^{1/k} \to 0$ as $k \to \infty$ when $\abs{z} < 1$ and that $\log\abs{1-z^k}^{1/k} \to \log \abs{z}$ as $k \to \infty$ when $1 < \abs{z}$.  On the unit circle, the limit does not exist.  In order to apply the Lebesgue dominated convergence theorem, we need to bound the integrand and this results in bounding away from the unit circle.  To this end, for $\nu > 1$, we separate the above integral into integrals over three regions: (i) $\abs{z} < \nu\inv$, (ii) $\nu\inv \leq \abs{z} \leq \nu$, and (iii) $\nu < \abs{z}$.  The integral over the first region thus limits to 0 and the integral over the third region limits to the integral of $\log \abs{z}$ over $\nu < \abs{z}$.  It is the integral over the region $\nu\inv \leq \abs{z} \leq \nu$ that requires further investigation.  We desire to show that this integral is equal to 0.   Using the work of Haagerup--Schultz~\cite{ar:HS09} on the invariant subspace problem, this integral is equal to $\frac{1}{k}$ times the logarithm of the Fuglede--Kadison determinant of the operator $I - L^k$ restricted to a certain invariant subspace $\frakK_\nu \subseteq L^2(G_\phi)^n$, see Theorem~\ref{th:subspaces}.  This is where the chain flare condition comes into play.  Under the chain flare condition we can identify this subspace for $\nu$ sufficiently close to and greater than 1.  There are two cases:
\begin{enumerate}

\item The subspace $\frakK_\nu$ is the trivial subspace.  In the language of the chain flare condition, this happens when there is no Nielsen 1--chain so that the quasi-fixed submodule $V_{\rm qf}$ is trivial.  By definition, the 0 morphism has Fuglede--Kadison determinant equal to 1 and so the integral over the region $\nu\inv \leq \abs{z} \leq \nu$ is 0 as desired.

\item The subspace $\frakK_\nu$ is isomorphic to $L^2(G_\phi)$ and the restriction of $I - L^k$ to this subspace is induced from the operator given by right multiplication by $1 - t^k$ on the subgroup $\I{t} \subset G_\phi$.  In the language of the chain flare condition, this happens when there is a Nielsen 1--chain that generates the quasi-fixed submodule $V_{\rm qf}$.  Hence, by properties of the Fuglede--Kadison determinant, the determinant of the restriction of $I - L^k$ to $\frakK_\nu$ equals the determinant of the operator given by right multiplication by $1 - t^k$ on $L^2(\I{t})$.  This operator has determinant equal to 1 and again so the integral over the region $\nu\inv \leq \abs{z} \leq \nu$ is 0 as desired.    

\end{enumerate}
Since this holds for all $\nu > 1$, we conclude that:
\begin{equation*}\label{eq:outline}
-\rho^{(2)}(G_\phi) = \int_{1 < \abs{z}} \log \abs{z} \, d\mu_L
\end{equation*} 
as claimed.

To complete the proof of Theorem~\ref{th:mahler}, we must show that the integral in the above equation is positive.   Using the properties of the Brown measure and the operator $L$, we show in proof of Theorem~\ref{th:mahler} that:
\begin{equation*}
0 \leq \int_{\abs{z} < 1} \log \abs{z} \, d\mu_L + \int_{1 < \abs{z}} \log \abs{z} \, d\mu_L
\end{equation*}
As the first integral is non-positive, the second integral is non-negative and as we can show that the support of the measure $d\mu_L$ is not contained in the unit circle, we conclude that the second integral is in fact positive.


\subsection{Organization of paper}\label{subsec:organization}

This paper is organized as follows.  In Section~\ref{sec:cfh} we introduce the concepts and notation necessary to state the chain flare condition, which is formally stated in Section~\ref{subsec:statement}.  Sections~\ref{subsec:von neumann} and \ref{subsec:torsion} define the notion of the Fuglede--Kadison determinant and the $L^2$--torsion, especially in the context of free-by-cyclic groups.  The author's previous work on computing this invariant using a topological representative, in particular the definition of the operators in Theorem~\ref{th:mahler} is recalled in Section~\ref{subsec:compute}.  The Brown measure associated to Hilbert--$G$--module morphism $A \from U \to U$ is introduced in Section~\ref{subsec:brown} and its relation to the Haagerup--Schultz invariant subspaces is explained in Section~\ref{subsec:HS}.  The work on using the chain flare condition to understand the invariant subspace $\frakK_\nu$ is initiated in Section~\ref{sec:dynamics quasi-fixed} where we explore the dynamics on quasi-fixed submodule $V_{\rm qf}$ mentioned in Section~\ref{subsec:outline}.  The are two cases, depending on whether the Nielsen 1--chain is non-geometric (Section~\ref{subsec:non-geometric}) or geometric (Section~\ref{subsec:geometric}).  In Sections~\ref{sec:isolating} and \ref{sec:restriction}, we identify the subspace $\frakK_\nu$ as explained in Section~\ref{subsec:outline} and compute the Fuglede--Kadison determinant of the restriction of the operator $I - L_{f,s}$ to this subspace.  The proof of Theorem~\ref{th:mahler} takes place in Section~\ref{sec:integral}.  We conclude in Section~\ref{sec:apply} with some final remarks regarding the chain flare condition and on applying the methods within to ascending HNN-extensions $\FF \ast_\Psi$.




\section{The chain flare condition}\label{sec:cfh}

The goal of this section is to give the complete statement of the chain flare condition.  There are several technicalities that are necessary to derive a statement that works for a general topological representative $f \from \Gamma \to \Gamma$ and that takes into account invariant subgraphs.  On a first read, the reader is invited in Sections~\ref{subsec:relative 1-chains} and \ref{subsec:statement} to assume that the graph $\Gamma$ is a rose and that the invariant subgraph $H \subset \Gamma$ is a single vertex.  In this case, all of the $\QQ[\FF]$--modules defined in Section~\ref{subsec:relative 1-chains} are equal and isomorphic to $\QQ[\FF]^n$ ($n$ is the rank of $\FF$) and the homomorphism $A_{f,H}$ is an isomorphism of this free module.  


\subsection{Graphs and morphisms}\label{subsec:graphs morphisms}

A \emph{graph} is a 1--dimensional $CW$--complex.  If $\Gamma$ is a graph, by $\sfV(\Gamma)$ we denote the set of \emph{vertices} (0--cells) and by $\sfE(\Gamma)$ we denote the set of \emph{edges} (1--cells).  As 1--cells, edges are oriented; the initial vertex of an edge $e \in \sfE(\Gamma)$ is denoted $\sfo(e)$ and the terminal vertex is denoted $\sft(e)$.  The same edge with opposite orientation is denoted by $\bar{e}$.  

An \emph{edge-path} is the image of a continuous map $\bp \from [0,1] \to \Gamma$ for which there exists a partition $0 = x_0 < x_1 < \cdots < x_m = 1$ such that $\bp|_{[x_{k-1},x_k]}$ is homeomorphism onto an edge of $\Gamma$.  When there is no ambiguity, we will define an edge-path by listing the vertices it visits and write $\bp\from p_0,\ldots,p_m$ where $p_k = \bp(x_k)$ or by listing the edges it visits.  

For a graph $\Gamma$, a \emph{morphism} $f \from \Gamma \to \Gamma$ is a cellular map that linearly expands each edge of $\Gamma$ across an edge-path in $\Gamma$ (with respect to some metric).  Fixing an enumeration of edges of $\Gamma$, $\sfE(\Gamma) = \{e_1,\ldots,e_n\}$, the \emph{transition matrix $M(f)$} is the $n \times n$ matrix where $m_{i,j}$ equals the number of occurrences of $e_j$ or $\bar{e}_j$ in the edge-path $f(e_i)$.  

The morphism $f \from \Gamma \to \Gamma$ \emph{respects} a filtration of $\Gamma$ by subgraphs $\emptyset = \Gamma_0 \subset \Gamma_1 \subset \cdots \subset \Gamma_S = \Gamma$ if $f(\Gamma_s) \subseteq \Gamma_s$ for all $1 \leq s \leq S$.  In this case, the transition matrix can be assumed to have a lower block triangular form.  Indeed this happens so long as edges lower in the filtration are ordered first, i.e, $e_i \in \sfE(\Gamma_s)$ and $e_j \notin \sfE(\Gamma_s)$ implies that $i < j$.  Let $i_s$ denote the smallest index with $e_{i_s} \notin \sfE(\Gamma_{s-1})$, let $n_s = \#\abs{\sfE(\Gamma_s) - \sfE(\Gamma_{s-1})}$ and let $M(f)_s$ denote the $n_s \times n_s$ submatrix of $M(f)$ with $i_s \leq i,j \leq i_s + n_s - 1$.  Then $M(f)$ is lower block triangular with the submatrices $M(f)_s$ along the diagonal.

Unless otherwise stated, we will always assume that such a filtration $\emptyset = \Gamma_0 \subset \Gamma_1 \subset \cdots \subset \Gamma_S = \Gamma$ is maximal in the sense that $M(f)_s$ is either the zero matrix or irreducible for each $1 \leq s \leq S$.  For each $1 \leq s \leq S$ where $M(f)_s$ is irreducible, we let $\lambda(f)_s$ denote the associated Perron--Frobenius eigenvalue.  We set $\PF(f) = \{ s \mid M(f)_s \text{ is irreducible and } \lambda(f)_s > 1 \}$.  We say the filtration $\emptyset = \Gamma_0 \subset \Gamma_1 \subset \cdots \subset \Gamma_S = \Gamma$ is \emph{reduced} if for each $1 \leq s \leq S$ there is exactly one component $\hGamma_s \subseteq \Gamma_s$ that is not a component of $\Gamma_{s-1}$.  (This definition is not the standard usage of the term reduced in the context of filtrations---cf.~\cite{ar:BFH00,ar:FH11,ar:HM20}, but our definition is easily implied by the standard usage and our definition is what we use in the sequel.)

We say a homotopy equivalence $f \from \Gamma \to \Gamma$ \emph{represents} an outer automorphism $\phi \in \Out(\FF)$ if there is a vertex $\ast \in \sfV(\Gamma)$, an identification $\pi_1(\Gamma,*) \cong \FF$ and an edge-path from $*$ to $f(*)$ such that the outer automorphism induced by $f$ and this edge-path is $\phi$.  Unless otherwise noted, all maps of graphs in the sequel are morphisms.   


\subsection{Relative 1--chains}\label{subsec:relative 1-chains}
Suppose $f \from \Gamma \to \Gamma$ is a homotopy equivalence that fixes a vertex $* \in \sfV(\Gamma)$ and that $H \subset \Gamma$ is an $f$--invariant subgraph.  Fixing an isomorphism $\pi_1(\Gamma,*) \cong \FF$ we have that $f$ and the trivial path based at $*$ induces an automorphism of $\FF$ that we denote by $\Phi_{f}$.  

Let $\tGamma$ be the universal cover of $\Gamma$ and let $\tH$ be the union of the lifts of $H$ to $\tGamma$.  Fix a lift $\tast$ of $*$ to $\tGamma$ and let $\tf$ be the lift of $f$ such that $\tf(\tast) = \tast$.  This map satisfies $\tf(gz) = \Phi_f(g)\tf(z)$ for any point $z \in \tGamma$ and any element $g \in \FF$.  The set of rational (cellular) 1--chains, $C_1(\tGamma;\QQ)$, is a $\QQ[\FF]$--module isomorphic to $\QQ[\FF]^n$, where $n$ is the number of edges in $\Gamma$.  We express $1$--chains as formal linear combinations of the edges in $\tGamma$ and write $x = \sum_{e \in \sfE(\tGamma)} x_e e$ where $x_e \in \QQ$ and $x_e \neq 0$ for only finitely many $e \in \sfE(\Gamma)$.  The \emph{support} of a 1--chain $x \in C_1(\tGamma;\QQ)$ is defined by $\supp(x) = \{ e \in \sfE(\tGamma) \mid x_e \neq 0 \}$.  Note that changing the orientation on $e$ swaps the sign of the corresponding coefficient.  The map $\tf$ induces an abelian group homomorphism $A_f \from C_1(\tGamma;\QQ) \to C_1(\tGamma;\QQ)$ that satisfies $A_f(gx) = \Phi_f(g)A_f(x)$ for any 1--chain $x \in C_1(\tGamma;\QQ)$ and any element $g \in \FF$.  

Similarly, we also consider the of rational (cellular) 0--chains $C_0(\tGamma;\QQ)$ and the usual boundary map $\bd_1 \from C_1(\tGamma;\QQ) \to C_0(\tGamma;\QQ)$ defined on edges $\bd_1 e = \sft(e) - \sfo(e)$.  

Given distinct vertices $u_1,u_2 \in \sfV(\tGamma)$, by $[u_1,u_2]$ we denote the 1--chain in $C_1(\tGamma;\QQ)$ uniquely determined by:
\begin{equation*}
\bd_1[u_1,u_2]_v = \begin{cases}
-1 & \mbox{ if } v = u_1 \\
1 & \mbox{ if } v = u_2 \\ 
0 & \mbox{ else}.
\end{cases}
\end{equation*}
In particular, $[u_1,u_2]_e = \pm 1$ for any edge in the edge-path from $u_1$ to $u_2$ and $[u_1,u_2]_e = 0$ for all other edges.

We consider the following $\QQ[\FF]$--submodule of rational 1--chains in $\tGamma$ relative to $\tH$:
\begin{align*}
C_1(\tGamma,\tH;\QQ) &= \{ x \in C_1(\tGamma;\QQ) \mid x_e = 0\ \forall e \in \sfE(\tH) \}.
\end{align*}
We observe that $C_1(\tGamma;\QQ) = C_1(\tH;\QQ) \oplus C_1(\tGamma,\tH;\QQ)$.  By $\pi_H$ and $\pi_H^\perp$ respectively we denote the projections of $C_1(\tGamma;\QQ)$ onto $C_1(\tH;\QQ)$ and $C_1(\tGamma;\tH;\QQ)$ respectively.  The following homomorphism is central to the chain flare condition:
\begin{equation*}
A_{f,H} = \pi_H^\perp \circ A_f\big|_{C_1(\tGamma,\tH;\QQ)} \from C_1(\tGamma,\tH;\QQ) \to C_1(\tGamma,\tH;\QQ).
\end{equation*}  


\subsection{Nielsen 1--chains}\label{subsec:nielsen 1-chain}

As stated in the Introduction, in essence, the chain flare condition states that the norm of a relative 1--chain in $C_1(\tGamma,\tH;\QQ)$ should grow by a definite factor after applying $A_{f,H}$ or else it is the image of a relative 1--chain whose norm is a larger by a definite factor.  However, there are certain 1--chains that are fixed by $A_{f,H}$ that need to be accounted for.  This is the motivation for the definition of a Nielsen 1--chain.

\begin{definition}\label{def:nielsen 1-chain}
Let $\rho \in C_1(\tGamma,\tH;\QQ)$ be a relative 1--chain such that $\rho = \pi_H^\perp([u,v])$ for some vertices $u,v \in \sfV(\tGamma)$ that are fixed by $\tf$, i.e., $\tf(u) = u$ and $\tf(v)= v$.  We say $\rho$ is a \emph{non-geometric Nielsen 1--chain} if it satisfies the following condition.
\begin{enumerate}[label=(NNC{{\arabic*}}),leftmargin=2cm]
\item\label{nnc:single} There is an edge $e \in \sfE(\tGamma) - \sfE(\tH)$ such that $\rho_e = \pm 1$ and $\rho_{ge} = 0$ for any non-trivial element $g \in \FF$.
\end{enumerate}
We say $\rho$ is a \emph{geometric Nielsen 1--chain} if it satisfies the following conditions.
\begin{enumerate}[label=(GNC{{\arabic*}}),leftmargin=2cm]
\item\label{gnc:single} For distinct elements $g_1,g_2 \in \FF$, the intersection $\supp(g_1\rho) \cap \supp(g_2\rho)$ is either empty or consists of a single edge.

\item\label{gnc:pair} For all edges $e \in \sfE(\tGamma) - \sfE(\tH)$, there are exactly two elements $g_1,g_2 \in \FF$ such that $e$ is the unique edge in the intersection $\supp(g_1\rho) \cap \supp(g_2\rho)$.

\item\label{gnc:noncommuting} There exists non-commuting elements $g_1,g_2 \in \FF$ such that the intersections $\supp(\rho) \cap \supp(g_1\rho)$ and $\supp(\rho)\cap \supp(g_2\rho)$ are non-empty.

\end{enumerate}
We say $\rho$ is a \emph{Nielsen 1--chain} if it is either a non-geometric or a geometric Nielsen 1--chain.
\end{definition}

We observe that $A_{f,H}(\rho) = \rho$.  In Section~\ref{sec:apply} we explain how Nielsen 1--chains naturally arise for EG strata in a CT map.


\subsection{The chain flare condition}\label{subsec:statement}

We require the following notation before we state the chain flare condition.  We consider the usual $L^2$--inner product and $L^2$--norm on 1--chains.  That is, given a 1--chains $x = \sum_{e \in \sfE(\tGamma)} x_e e$ and $x'  = \sum_{e \in \sfE(\tGamma)} x'_e e$ we set:
\begin{equation*}
\I{x,x'} = \sum_{e \in \sfE(\tGamma)} x_e x'_e \mbox{ and } \norm{x}^2 = \I{x,x} = \sum_{e \in \sfE(\tGamma)} \abs{x_e}^2.
\end{equation*}  

If $V \subseteq C_1(\tGamma;\tH;\QQ)$ is a $\QQ[\FF]$--submodule and $0 < \theta < 1$, we set:
\begin{equation*}
N_\theta(V) = \left\{ x' \in C_1(\tGamma,\tH;\QQ) \mid \I{x,x'} > \theta \norm{x}\norm{x'} \mbox{for some } x \in V \right\}.
\end{equation*}
Thus, $N_\theta(V)$ consists of elements that make a small angle with an element of $V$.  Notice that $N_\theta(V) - \{0\}$ is a neighborhood of $V - \{0\}$.

If $N \subseteq C_1(\tGamma,\tH;\QQ)$ is a subset, we define the following subset: 
\begin{equation*}
N^{\infty} = \bigcup_{k \in \ZZ} A_{f,H}^{k}(N).
\end{equation*}
In other words, $N^{\infty}$ consists of all relative 1--chains $x'$ such that either $x' = A_{f,H}^k(x)$ for some $x \in V$ and some $k \geq 0$ or that $A^k_{f,H}(x') \in V$ for some $k \geq 0$.  We remark that $N^{\infty}$ is $A_{f,H}$--invariant.

We can now formally state the chain flare condition.  

\medskip 

\noindent {\bf Chain Flare Condition.} Suppose $f \from \Gamma \to \Gamma$ is a homotopy equivalence and $H \subset \Gamma$ is an $f$--invariant subgraph.  We say \emph{$f$ satisfies the chain flare condition relative to $H$} if there are $\QQ[\FF]$--submodules $V_{\rm h}, V_{\rm qf} \subseteq C_1(\tGamma, \tH; \QQ)$ where the following conditions hold.
\begin{enumerate}[label=(CFH{{\arabic*}}),leftmargin=2cm]
\item\label{cfh:1} $C_1(\tGamma,\tH;\QQ) = V_{\rm h} + V_{\rm qf}$.

\item\label{cfh:2} There exists constants $\lambda > 1$ and $0 < \theta < 1$ such that for all $x \in N_\theta(V_{\rm h})^{\infty}$:
\begin{equation*}
\lambda\norm{A_{f,H}(x)} \leq \max\left\{\norm{A^2_{f,H}(x)},\norm{x}\right\}.
\end{equation*}

\item\label{cfh:3} If $V_{\rm qf} \neq \{0\}$, then there exists a Nielsen 1--chain $\rho \in C_1(\tGamma,\tH;\QQ)$ such that for all $x \in V_{\rm qf}$, there exist rational numbers $q_1,\ldots,q_r \in \QQ$ and elements $g_1,\ldots,g_r \in \FF$ such that $x = q_1g_1\rho + \cdots + q_rg_r\rho$.

\end{enumerate}

\medskip 

We call $V_{\rm h}$ the \emph{hyperbolic submodule} and $V_{\rm qf}$ the \emph{quasi-fixed submodule}.  If $V_{\rm qf} \neq \{0\}$, we say the Nielsen 1--chain $\rho$ specified in \ref{cfh:3} \emph{generates} the submodule.  We remark that to verify \ref{cfh:2}, one may assume that the coefficients in $x$ are integral.  Further, if $V_{\rm qf} = \{0\}$, then it suffices to verify \ref{cfh:2} only for $x \in V_{\rm h}$ as $N_\theta(V_{\rm h})^\infty$ equals $V_{\rm h}$ in this case.  Moreover, if $C_1(\tGamma,\tH;\QQ) = V_{\rm h} \oplus V_{\rm qf}$ and $V_{\rm h}$ is $A_{f,H}$--invariant, then it suffices to verify \ref{cfh:2} only for $x \in V_{\rm h}$.  See Remark~\ref{rem:invariant}.

For use later on in Section~\ref{sec:isolating}, we record the following consequence of \ref{cfh:2}.

\begin{lemma}\label{lem:induction}
Suppose that the homotopy equivalence $f \from \Gamma \to \Gamma$ satisfies the chain flare condition relative to the $f$--invariant graph $H \subset \Gamma$ with constants $\lambda$ and $\theta$.  The following statements hold.

\begin{enumerate}
\item\label{induction down} If $x \in N_\theta(V_{\rm h})^{\infty}$, $j \geq 1$ and $\lambda\norm{A^j_{f,H}(x)} \leq \norm{A_{f,H}^{j-1}(x)}$, then $\lambda^j\norm{A^j_{f,H}(x)} \leq \norm{x}$.

\item\label{induction up} If $x \in N_\theta(V_{\rm h})^{\infty}$, $j \geq 1$ and $\lambda\norm{x} \leq \norm{A_{f,H}(x)}$, then $\lambda^j\norm{x} \leq \norm{A_{f,H}^j(x)}$.

\item\label{induction power} If $x \in N_\theta(V_{\rm h})^{\infty}$ and $N \geq 1$, then:
\begin{equation*}
\lambda^N\norm{A_{f,H}^N(x)} \leq \max\left\{\norm{A^{2N}_{f,H}(x)},\norm{x}\right\}.
\end{equation*}
\end{enumerate}
\end{lemma}

\begin{proof}
To simplify notation, we denote $A_{f,H}$ by $A$ in the proof.

We first prove \eqref{induction down} by induction.  The statement is tautological for $j = 1$.  Now suppose that $j \geq 2$, that \eqref{induction down} holds for $j-1$, that $x \in N_\theta(V_{\rm h})^\infty$ and that $\lambda\norm{A^j(x)} \leq \norm{A^{j-1}(x)}$.  Since $A^{j-2}(x) \in N_\theta(V_{\rm h})^{\infty}$, by \ref{cfh:2} we must have:
\begin{equation*}
\lambda\norm{A^{j-1}(x)} \leq \max\left\{\norm{A^j(x)},\norm{A^{j-2}(x)}\right\}.
\end{equation*}
As $\lambda\norm{A^j(x)} \leq \norm{A^{j-1}(x)}$ by assumption, we must have $\lambda\norm{A^{j-1}(x)} \leq \norm{A^{j-2}(x)}$ as $\lambda > 1$.  Hence, by induction $\lambda^{j-1}\norm{A^{j-1}(x)} \leq \norm{x}$.  Therefore:
\begin{equation*}
\lambda^j\norm{A^j(x)} = \lambda^{j-1}\left(\lambda\norm{A^j(x)}\right)\leq \lambda^{j-1}\norm{A^{j-1}(x)} \leq \norm{x}.
\end{equation*}

The proof of \eqref{induction up} is similar.  We provide the details for completeness.  Again, the statement is tautological for $j = 1$.  Suppose that $j \geq 2$, that  \eqref{induction up} holds for $j-1$, that $x \in N_\theta(V_{\rm h})^\infty$ and that $\lambda \norm{x} \leq \norm{A(x)}$.  Since $A(x) \in N_\theta(V_{\rm h})^{\infty}$, by \ref{cfh:2} we must have:
\begin{equation*}
\lambda\norm{A(x)} \leq \max\left\{\norm{A^2(x)},\norm{x}\right\}.
\end{equation*}
Thus, as before we find that $\lambda\norm{A(x)} \leq \norm{A^2(x)}$.  Hence, by induction $\lambda^{j-1}\norm{A(x)} \leq \norm{A^{j-1}A(x)} = \norm{A^j(x)}$.  Therefore:
\begin{equation*}
\lambda^j\norm{x} = \lambda^{j-1}\left(\lambda\norm{x}\right) \leq \lambda^{j-1}\norm{A(x)} \leq \norm{A^j(x)}.\qedhere
\end{equation*}

We now consider \eqref{induction power}.  By \ref{cfh:2}, we have that $\lambda\norm{A^N(x)} \leq \max\left\{\norm{A^{N+1}(x)},\norm{A^{N-1}(x)}\right\}$.  If $\lambda\norm{A^N(x)} \leq \norm{A^{N+1}(x)}$, then by \eqref{induction up} we have $\lambda^N\norm{A^N(x)} \leq \norm{A^{2N}(x)}$.  If $\lambda\norm{A^N(x)} \leq \norm{A^{N-1}(x)}$, then by \eqref{induction down} we have $\lambda^N\norm{A^N(x)} \leq \norm{x}$.  This proves \eqref{induction power}.  
\end{proof}


\section{\texorpdfstring{$L^{2}$}{L\textasciicircum2}--torsion of free-by-cyclic groups}\label{sec:torsion}

In this section, we recall the definition of $L^2$--torsion of a free-by-cyclic group as well as some results necessary for the sequel.  General references for the material in this section are the survey paper by Eckmann~\cite{ar:Eckmann00} and the book by L\"uck~\cite{bk:Luck02}.


\subsection{The von Neumann algebra of a countable group}\label{subsec:von neumann}

Let $G$ be a countable group.  By $L^2(G)$ we denote the vector space of square summable functions $\xi \from G \to \CC$.  We will express an element of $L^2(G)$ as a formal linear combination $\xi = \sum_{g \in G} \xi_g g$ where $\xi_g \in \CC$ and $\sum_{g \in G} \abs{\xi_g}^2 < \infty$.  This is a Hilbert space with inner product:
\begin{equation*}
\I{\xi,\xi'} = \sum_{g \in G} \xi_g\overline{\xi'_g}.
\end{equation*}
The associated norm is denoted $\norm{\xi} = \I{\xi,\xi}^{1/2}$.  The dense subspace of finitely supported functions is isomorphic (as a vector space) to the group algebra $\CC[G]$ and as such we consider $\CC[G]$ as a subspace of $L^2(G)$.  The group $G$ acts isometrically on both the left and the right of $L^2(G)$ where for $h \in G$ and $\xi \in L^2(G)$ we define:
\begin{equation*}
h\cdot\xi = \sum_{g \in G} \xi_g hg = \sum_{g \in G} \xi_{h\inv g} g \mbox{ and } \xi \cdot h = \sum_{g \in G} \xi_g gh = \sum_{g \in G} \xi_{gh\inv} g.
\end{equation*}  
By linearity, these extend to actions on $L^2(G)$ of the group algebra $\CC[G]$ by bounded operators. In the sequel, when we say that some function or object related to $L^2(G)$ is $G$--equivariant or $G$--invariant, we are referring to the left action.  

The \emph{von Neumann algebra} of $G$, denoted $\calN(G)$, is the algebra of $G$--equivariant bounded operators on $L^2(G)$.  That is, an element $A \in \calN(G)$ is a bounded operator $A \from L^2(G) \to L^2(G)$ such that $A(g \cdot \xi) = g \cdot A(\xi)$ for all $g \in G$ and $\xi \in L^2(G)$.  In particular, for each $x \in \CC[G]$, the operator $A_x \from L^2(G) \to L^2(G)$ defined by $A_x(\xi) = \xi \cdot x$ is $G$--equivariant and hence we can consider $\CC[G]$ as a subalgebra of $\calN(G)$.     

There is a notion of \emph{trace} for elements in $\calN(G)$ that is defined by:
\begin{equation*}
\Tr_G(A) = \I{A(\id_G),\id_G}
\end{equation*}
where $\id_G$ is the identity element of $G$.  More generally, a $G$--equivariant bounded operator $A \from L^2(G)^n \to L^2(G)^n$ can be expressed as a matrix $A = [A_{i,j}]$ where each $A_{i,j} \in \calN(G)$ and we define:
\begin{equation*}
\Tr_G(A) = \sum_{i=1}^{n} \Tr_G(A_{i,i}).
\end{equation*}

A \emph{(finitely generated) Hilbert--$G$--module} is a Hilbert space $U$ that admits an isometric action by $G$ and for which there exists a $G$--equivariant isometric embedding $U \to L^2(G)^n$ for some $n$.  The notion of trace allows for the definition of dimension of a Hilbert--$G$--module by:
\begin{equation*}
\dim_G(U) = \Tr_G(P_U)
\end{equation*}
where $P_U \from L^2(G)^n \to L^2(G)^n$ is the projection onto the image of $U$.  

A \emph{morphism} of Hilbert--$G$--modules $U$ and $V$ is a $G$--equivariant bounded operator $A \from U \to V$.  For a morphism $A \from U \to V$ of Hilbert--G--modules, we denote by $F_A \from [0,\infty) \to [0,\infty)$ the \emph{spectral density function} of $A$, that is, $F_A(\lambda) = \Tr_G \left(E^{A^* A}_{\lambda^2}\right)$ where $\left\{ E^{A^* A}_{\lambda}\right\}$ is the spectral family of $A^* A$.  The \emph{Fuglede--Kadison determinant of $A$} is defined by:
\begin{equation*}
\detG (A) = \exp \int_{0^+}^\infty \log (\lambda) \, dF_A
\end{equation*} 
if the integral exists, and $\det_G(A)$ is defined to be $0$ otherwise.

We record the following property of the Fuglede--Kadison determinant for later use.

\begin{lemma}\label{lem:det isomorphism}
Let $A \from U \to U$ and $O \from U \to V$ be morphisms of finite dimensional Hilbert--$G$--modules where $A$ is injective and $O$ is an isomorphism.  Then:
\begin{equation*}
\detG (A) = \detG (OAO\inv).
\end{equation*}
\end{lemma}

\begin{proof}
Let $I$ denote the identity operator $I \from U \to U$.  By \cite[Theorem~3.14~(1)]{bk:Luck02} we have that $1 = \detG (I) = \detG (OO\inv) = \detG (O) \cdot \detG (O\inv)$.  Hence, by \cite[Theorem~3.14~(1)]{bk:Luck02} again, we find:
\begin{equation*}
\detG (OAO\inv) = \detG(O) \cdot \detG(A) \cdot \detG(O\inv) = \detG (A).\qedhere
\end{equation*}
\end{proof}


\subsection{\texorpdfstring{$L^{2}$}{L\textasciicircum2}--torsion of free-by-cyclic groups}\label{subsec:torsion}

Let $G$ be a countable group and let $X$ be a $CW$--complex that admits a continuous action by $G$ that freely permutes the cells of $X$ and such that there are only finitely many $G$--orbits of cells.  The cellular chain complex $C_*(X) = \{\bd_j \from C_j(X) \to C_{j-1}(X)\}$ consists of free $\ZZ[G]$--modules of finite rank.  Thus $C_*^{(2)}(X) = L^2(G) \otimes_{\ZZ[G]} C_*(X)$ is a chain complex of Hilbert--$G$--modules.  If $C_*^{(2)}(X)$ is weakly acyclic, i.e., $\ker \bd_j = \clos(\img \bd_{j+1})$ for all $j$, and $\detG (\bd_j) \neq 0$ for all $j$, then the \emph{$L^2$--torsion} of $X$ is defined by:
\begin{equation*}
\rho^{(2)}(X) = -\sum_{j \geq 0} (-1)^j \log \detG (\bd_j). 
\end{equation*} 

The situation we are most often interested in is when $G$ has a finite classifying space, $BG$, and $X = EG$.  There is a large class of groups for which the chain complex of Hilbert--$G$--modules $\{\bd_j \from C^{(2)}_j(EG) \to C^{(2)}_{j-1}(EG)\}$ satisfies the above assumptions and moreover the $L^2$--torsion of $EG$ depends only on $G$ and not the particular choice of $BG$~\cite[Lemma~13.6]{bk:Luck02}.  This class of groups includes free-by-cyclic groups in particular and thus we are justified in defining
\begin{equation*}
\rho^{(2)}(G_\phi) = \rho^{(2)}(EG_\phi).
\end{equation*}

This invariant behaves in some ways like Euler characteristic.  The following property illustrates this connection and is used later on.

\begin{theorem}[{\cite[Theorem~7.27~(4)]{bk:Luck02}}]\label{th:torsion powers}
Suppose $\phi$ is an outer automorphism of $\FF$.  Then for all $k \geq 1$:
\[ \rho^{(2)}(G_{\phi^k}) = k\rho^{(2)}(G_\phi).  \]
\end{theorem}


\subsection{Computing the \texorpdfstring{$L^{2}$}{L\textasciicircum2}--torsion from a topological representative}\label{subsec:compute}

In the remainder of this section, we briefly explain how to compute the $L^2$--torsion $-\rho^{(2)}(G_\phi)$ from a homotopy equivalence $f \from \Gamma \to \Gamma$ that represents $\phi \in \Out(\FF)$.  See \cite[Section~4]{ar:Clay17-2} for complete details.  As in Section~\ref{subsec:relative 1-chains}, we assume that $f$ fixes a vertex $* \in \sfV(\Gamma)$ and fixing an isomorphism $\pi_1(\Gamma,*) \cong \FF$, we let $\Phi_f$ denote the automorphism induced by $f$ and the trivial path based at $*$.  We will use the semi-direct product presentation $\FF \rtimes_{\Phi_f} \I{t}$ for the corresponding free-by-cyclic group $G_\phi$. Let $\tf \from \tGamma \to \tGamma$ be the corresponding lift of $f$ to the universal cover and let $A_f \from C_1(\tGamma;\QQ) \to C_1(\tGamma;\QQ)$ be the corresponding homomorphism.   

Let $X_f$ be the mapping torus of $f$, that is:
\begin{equation*}
X_f = \raisebox{5pt}{$\Gamma \times [0,1]$} \Big/ \raisebox{-5pt}{$(x,0) \sim (f(x),1)$},
\end{equation*}
and let $\tX_f$ the universal cover of $X_f$.  An edge in $\tX_f$ is called \emph{horizontal} if it is the lift of an edge in $\Gamma \times \{0\} \subset X_f$ and \emph{vertical} otherwise.  The subspace of (cellular) 1--chains $C_1(\tX_f)$ spanned by horizontal edges is a free $\ZZ[G_\phi]$--module of rank $n = \#\abs{\sfE(\Gamma)}$.  Likewise, the set of (cellular) 2--chains $C_2(\tX_f)$ is also a free $\ZZ[G_\phi]$--module of rank $n = \#\abs{\sfE(\Gamma)}$.  Hence, after choosing appropriate bases, the cellular boundary map $\bd_2 \from C_2(\tX_f) \to C_1(\tX_f)$ followed by projection to the subspace spanned by the horizontal edges determines a $\ZZ[G_\phi]$--module homomorphism:
\begin{equation*}
\bd_{\rm hor} \from \ZZ[G_\phi]^n \to \ZZ[G_\phi]^n
\end{equation*}
that is given by right multiplication by a matrix of the form $I - tJ_1(f)$ where $I$ is the identity matrix and $J_1(f) \in \Mat_{n}(\ZZ[\FF])$, which is the so-called \emph{Jacobian}.  Indeed, each 2--cell in $\tX_f$ has a unique bottom edge $e$ and the top edges are none other than $t\tf(e)$ so that the horizontal components of the boundary of this 2--cell is $e - t\tf(e)$.  See Figure~\ref{fig:2cell}.  The rows of $J_1(f)$ just record the edges in $A_{f}(e)$, using the isomorphism between $C_1(\tGamma)$ and $\ZZ[\FF]^n$.  Let $L_f \from L^2(G_\phi)^n \to L^2(G_\phi)^n$ be the operator given by right multiplication by $tJ_1(f)$ so that $\bd_{\rm hor} = I - L_f$ where $I$ is now consider as the identity operator.  Then, if the image of every vertex in $\Gamma$ is fixed by $f$, it was shown that (cf.~\cite[Theorem~7.29]{bk:Luck02}):
\begin{equation*}
-\rho^{(2)}(G_\phi) = \log \detG[G_\phi] \left( I - L_f \right).
\end{equation*} 

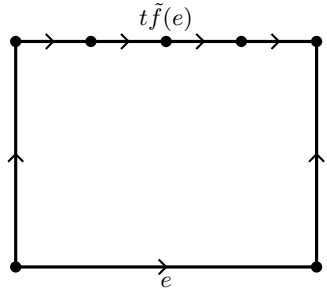
\begin{figure}
\centering
\begin{tikzpicture}
\draw[very thick] (0,0) rectangle (4,3);
\fill (0,0) circle [radius=0.075];
\fill (4,0) circle [radius=0.075];
\fill (0,3) circle [radius=0.075];
\fill (4,3) circle [radius=0.075];
\fill (1,3) circle [radius=0.075];
\fill (2,3) circle [radius=0.075];
\fill (3,3) circle [radius=0.075];
\draw[thick] (1.9,0.1) -- (2,0) -- (1.9,-0.1);
\draw[thick] (0.4,3.1) -- (0.5,3) -- (0.4,2.9);
\draw[thick] (1.4,3.1) -- (1.5,3) -- (1.4,2.9);
\draw[thick] (2.4,3.1) -- (2.5,3) -- (2.4,2.9);
\draw[thick] (3.4,3.1) -- (3.5,3) -- (3.4,2.9);
\draw[thick] (-0.1,1.4) -- (0,1.5) -- (0.1,1.4);
\draw[thick] (3.9,1.4) -- (4,1.5) -- (4.1,1.4);
\node[below] at (2,0) {\footnotesize $e$};
\node[above] at (2,3) {\footnotesize $t\tf(e)$};
\end{tikzpicture}
\caption{A 2--cell in $\tX_f$.}\label{fig:2cell}
\end{figure}

When $f \from \Gamma \to \Gamma$ respects a filtration $\emptyset = \Gamma_{0} \subset \Gamma_{1} \subset \cdots \subset \Gamma_{S} = \Gamma$ we can break up the above formula into pieces associated to the filtration elements.  To this end, for each $1 \leq s \leq S$, we consider the subcomplex $X_{f,s} \subseteq X_f$ which is the mapping torus of the restriction of $f$ to $\Gamma_s$.  We note that this subcomplex is not necessarily connected.  However, when the filtration is reduced, then there is exactly one component of $X_{f,s}$ that is not a component of $X_{f,s-1}$.  Let $\tX_{f,s}$ be the union of the lifts of $X_{f,s}$ to $\tX_f$.  The subspace of $C_1(\tX_f)$ spanned by the horizontal edges that lie in $\tX_{f,s} - \tX_{f,s-1}$ is a free $\ZZ[G_\phi]$--module of rank $n_s = \#\abs{\sfE(\Gamma_s) - \sfE(\Gamma_{s-1})}$.  Hence as above, after choosing appropriate bases, the cellular boundary map $\bd_2 \from C_2(\tX_{f,s} - \tX_{f,s-1}) \to C_1(\tX_{f,s})$ followed by projection to the subspace spanned by horizontal edges in $\tX_{f,s} - \tX_{f,s-1}$ determines a $\ZZ[G_\phi]$--module homomorphism:
\begin{equation*}
\bd_{{\rm hor}, s} \from \ZZ[G_\phi]^{n_s} \to \ZZ[G_\phi]^{n_s}
\end{equation*}
that is given by right multiplication by a matrix of the form $I - tJ_1(f)_s$ where $J_1(f)_s \in \Mat_{n_s}(\ZZ[\FF])$.  In terms of matrices, $J_1(f)$ is lower block triangular with blocks $J_1(f)_s$ on the diagonal, as was the connection between $M(f)$ and $M(f)_s$.  Let $L_{f,s} \from L^2(G_\phi)^{n_s} \to L^2(G_\phi)^{n_s}$ by the operator given by right multiplication by $tJ_1(f)_s$ so that $\bd_{{\rm hor}, s} = I - L_{f,s}$.  Using the lower block triangular form of $J_1(f)$, which comes the lower block triangular form of $M(f)$, the following theorem was shown.   

\begin{theorem}[{\cite[Theorem~4.10~\&~Remark~4.12]{ar:Clay17-2}}]\label{th:det-splitting}
Suppose that $f \from \Gamma \to \Gamma$ is a homotopy equivalence that respects the filtration $\emptyset = \Gamma_0 \subset \Gamma_1 \subset \cdots \subset \Gamma_S = \Gamma$, that $f \from \Gamma \to \Gamma$ represents the outer automorphism $\phi \in \Out(\FF)$ and that the image of each vertex in $\Gamma$ is fixed by $f$.  Then:
\begin{equation*}\label{eq:det-splitting}
-\rho^{(2)}(G_{\phi}) = \sum_{s \in \PF(f)} \log \detG[G_\phi] \bigl(I - L_{f, \, s}\bigr).
\end{equation*}
\end{theorem}

In the context of the chain flare condition we will use the following notation.  Let $f \from \Gamma \to \Gamma$ be a homotopy equivalence that represents $\phi \in \Out(\FF)$ and $H \subset \Gamma$ an $f$--invariant subgraph.  We consider the filtration (which is not necessarily maximal nor reduced) $\emptyset = \Gamma_0 \subset \Gamma_1 \subset \Gamma_2 = \Gamma$ where $\Gamma_1 = H$.  Then we set $n_H$ to be equal to $n_2$, the number of edges in $\Gamma - H$, and set $L_{f,H} \from L^2(G_\phi)^{n_H} \to L^2(G_\phi)^{n_H}$ to the be operator $L_{f,2}$.


\section{Brown measure and Haagerup--Schultz invariant subspaces}\label{sec:brown}

In this section we introduce the Brown measure $\mu_A$ for a $G$--equivariant bounded operator $A \from L^2(G)^n \to L^2(G)^n$, state its relation to the Fuglede--Kadison determinant and list the key properties that we require for the sequel.  Additionally, we introduce the Haagerup--Schultz invariant subspaces $\frakE(A,\nu)$ and $\frakF(A,\nu)$ associated to bounded operator on $A \from L^2(G)^n \to L(G)^n$ and state their relation to the Brown measure in the previously mentioned case when $A$ is $G$--equivariant.  The most important result of this section is Theorem~\ref{th:subspaces} which is essential for the proof of Theorem~\ref{th:mahler}.  The results in this section hold for more general von Neumann algebras but are stated in setting in which they will be applied within.  


\subsection{Brown measure}\label{subsec:brown}

Let $G$ be a countable group and $U$ a Hilbert--$G$--module.  Associated to a morphism $A \from U \to U$ is a Borel measure on $\CC$, called the \emph{Brown measure} and denoted $\mu_A$~\cite{col:Brown86}.  This measure maybe considered as giving the multiplicity of the values of the spectrum of $A$.  Indeed, if $G$ is a finite group, then $U$ is isomorphic as a vector space to $\CC^{n\abs{G}}$ for some $n$ and considering $A$ as an element of $\Mat_{n\abs{G}}(\CC)$ we have:
\begin{equation*}
\mu_A = \frac{1}{\abs{G}} \sum_{j = 1}^{n\abs{G}} \delta_{\lambda_j}
\end{equation*} 
where $\lambda_1,\ldots,\lambda_{n\abs{G}}$ are the eigenvalues of $A$ listed with multiplicity and $\delta_\lambda$ is the Dirac measure concentrated on the complex number $\lambda$.  

We summarize some of the properties of the Brown measure and its relation to the Fuglede--Kadison in the following theorem.

\begin{theorem}[{\cite[Theorem~3.13]{col:Brown86}}]\label{th:brown}
Let $A \from U \to U$ be a morphism of Hilbert--$G$--modules.  The following properties hold.
\begin{enumerate}
\item\label{brown support} The support of $\mu_A$ is contained in the spectrum of $A$.

\item\label{brown measure} $\mu_A(\CC) = \dim_G (U)$.

\item\label{brown integral} If $h \from \CC \to \CC$ is holomorphic, then
\begin{equation*}
\log \detG(h(A)) = \int_{\CC} \log \abs{h(z))} \, d\mu_{A}.
\end{equation*}

\end{enumerate}
\end{theorem}


\subsection{Haagerup--Schultz invariant subspaces}\label{subsec:HS}

In their study of the invariant subspace problem for operators in a type $\Pi_1$--factor, Haagerup--Schultz identified the following subspaces associated to a bounded operator on a Hilbert space.

\begin{definition}[{\cite[Definition~3.1 \& Lemma~3.2]{ar:HS09}}]\label{def:subspaces}
Let $A\from \calH \to \calH$ be a bounded operator on a Hilbert space.  For $\nu > 0$ we define the following $A$--invariant closed subspaces of $\calH$:
\begin{align*}
\frakE(A,\nu) &= \left\{ \xi \in \calH \mid \exists (\xi_{j}) \subset \calH \mbox{ with } \lim_{j \to \infty} \norm{\xi_j - \xi} = 0 \mbox{ and } \limsup_{j \to \infty} \norm{A^{j}\xi_{j}}^{1/j} \leq \nu \right\} \\[5pt]
\frakF(A,\nu) &= \left\{ \xi \in \calH \mid \exists (\xi_{j}) \subset \calH \mbox{ with } \lim_{j \to \infty} \norm{A^{j}\xi_{j} - \xi} = 0 \mbox{ and } \limsup \norm{\xi_{j}}^{1/j} \leq \nu\inv \right\}
\end{align*}
\end{definition}

\begin{remark}[{\cite[Remark~3.3]{ar:HS09}}]\label{rem:invertible}
If $A \from \calH \to \calH$ is invertible, then:
\begin{equation*}
\frakF(A,\nu) = \frakE(A\inv,\nu\inv) = \left\{ \xi \in \calH \mid \exists (\xi_{j}) \subset \calH \mbox{ with } \lim_{j \to \infty} \norm{\xi_j - \xi} = 0 \mbox{ and } \limsup_{j \to \infty} \norm{A^{-j}\xi_{j}}^{1/j} \leq \nu\inv \right\}
\end{equation*}
\end{remark}

There is a deep connection between these subspaces and the Brown measure $\mu_A$ when $A \from U \to U$ is a morphism of Hilbert--$G$--modules.  In particular, Haagerup--Schultz prove that the dimension of $\frakE(A,\nu)$ is the Brown measure of $\{z \in \CC \mid \abs{z} \leq \nu \}$ and the dimension of $\frakF(A,\nu)$ is the Brown measure of $\{z \in \CC \mid \abs{z} \geq \nu \}$~\cite[Lemma~7.18]{ar:HS09}.

For our purposes, we use these subspaces to isolate in the unit circle in the integral representation of $\log \detG (I - A^k)$.

\begin{theorem}\label{th:subspaces}
Let $A \from U \to U$ be a morphism of Hilbert--$G$--modules and let $\mu_A$ denote the Brown measure of $A$.  For $\nu > 1$, we set $\frakK_\nu = \frakE(A,\nu) \cap \frakF(A,\nu\inv)$. Then for $k \geq 1$ we have:
\begin{equation*}
\log \detG \bigl(I - A^k\bigr)\big|_{\frakK_\nu} = \int_{\nu\inv \leq \abs{z} \leq \nu} \log \abs{1 - z^k} \, d\mu_A.
\end{equation*}
\end{theorem}

\begin{proof}
Using the function $h(z) = 1 - z^k$, by Theorem~\ref{th:brown}~\eqref{brown integral} we have:
\begin{equation*}
\log \detG \bigl(I - A^k\bigr)\big|_{\frakK_\nu} = \int_{\CC} \log \abs{1 - z^k} \, d\mu_{A|_{\frakK_\nu}}.
\end{equation*}
Let $P \from U \to U$ denote the projection to the orthogonal complement of $\frakK_\nu$.  Then $\mu_A = \mu_{A|_{\frakK_\nu}} + \mu_{PAP}$ (cf.~\cite[Remark~7.17]{ar:HS09}).  Let $\CC_\nu = \{z \in \CC \mid \nu\inv \leq \abs{z} \leq \nu\}$.  According to \cite[Main~Theorem~1.1]{ar:HS09}, we have $\supp(\mu_{A|_{\frakK_\nu}}) \subseteq \CC_\nu$ and $\supp(\mu_{PAP}) \subseteq \CC - \CC_\nu$.  Therefore we have:
\begin{align*}
\log \detG \bigl(I - A^k\bigr)\big|_{\frakK_\nu} & = \int_{\CC} \log \abs{1 - z^k} \, d\mu_{A|_{\frakK_\nu}} \\
&= \int_{\nu\inv \leq \abs{z} \leq \nu} \log \abs{1 - z^k} \, d\mu_{A|_{\frakK_\nu}} \\
&= \int_{\nu\inv \leq \abs{z} \leq \nu} \log \abs{1 - z^k} \, d\mu_{A}.\qedhere
\end{align*}
\end{proof}

%
%
%
%
%
%


\section{Dynamics on the quasi-fixed submodule}\label{sec:dynamics quasi-fixed}

Using the setting and notation from Section~\ref{subsec:relative 1-chains}, we define the following Hilbert--$\FF$--modules:
\medskip
\begin{center}
$C^{(2)}_1(\tGamma,\tH) = L^2(\FF) \otimes_{\QQ[\FF]} C_1(\tGamma,\tH;\QQ), \quad V_{\rm h}^{(2)} = L^2(\FF) \otimes_{\QQ[\FF]} V_{\rm h}$ \\[7.5pt] 
and  $V_{\rm qf}^{(2)} = L^2(\FF) \otimes_{\QQ[\FF]} V_{\rm qf}$.
\end{center}
\medskip
We have that $C_1^{(2)}(\tGamma,\tH) = V^{(2)}_{\rm h} + V^{(2)}_{\rm qf}$. The homomorphism $A_{f,H} \from C_1(\tGamma,\tH;\QQ) \to C_1(\tGamma,\tH,\QQ)$ extends to a bounded operator $A_{f,H} \from C_1^{(2)}(\tGamma,\tH) \to C_1^{(2)}(\tGamma,\tH)$.  We will use a fixed isomorphism $C^{(2)}_1(\tGamma,\tH) \cong L^2(\FF)^{n_H}$ using a basis as in Section~\ref{subsec:compute} and consider $V^{(2)}_{\rm h}$ and $V^{(2)}_{\rm qf}$ as submodules of $L^2(\FF)^{n_H}$.  

In this section we explore the dynamics of $A_{f,H}$ on $V_{\rm qf}^{(2)}$.  We remark that $A_{f,H}$ is not $\FF$--equivariant.  The main result is the following.

\begin{theorem}\label{th:quasi fixed dynamics}
Suppose that the homotopy equivalence $f \from \Gamma \to \Gamma$ satisfies the chain flare condition relative to the $f$--invariant graph $H \subset \Gamma$.  Then there is a constant $C > 0$ such that for any $\xi \in V_{\rm qf}^{(2)}$ and $k \geq 0$ we have $C\inv\norm{\xi} \leq \norm{A_{f,H}^{k}(\xi)} \leq C\norm{\xi}$.
\end{theorem}

The only item of the chain flare condition that is needed for Theorem~\ref{th:quasi fixed dynamics} is \ref{cfh:3}.  The estimate in Theorem~\ref{th:quasi fixed dynamics} is used in the next section in the proof of Theorem~\ref{th:intersection} to show the equality between $V^{(2)}_{\rm qf}$ and the intersection $\frakE(A_{f,H},\nu) \cap \frakF(A_{f,H},\nu\inv)$ for $\nu$ sufficient close to and greater than $1$ when $f$ satisfies the chain flare condition.  

There are two cases to consider based on whether or not the Nielsen 1--chain generating $V_{\rm qf}$ is non-geometric or geometric (as the theorem obviously holds when $V_{\rm qf} = \{0\}$).  These two cases are proved in Section~\ref{subsec:non-geometric} (Proposition~\ref{prop:quasi fixed dynamics nongeom}) and Section~\ref{subsec:geometric} (Proposition~\ref{prop:quasi fixed dynamics geom}) respectively.  The key idea in both sections is to bound $\norm{A^k_{f,H}(x)}$ for $x \in V_{\rm qf}$ and $k \geq 0$ in terms of the rational coefficients used to express $x$ as a linear combination of translates of $\rho$ independent of $k$.    


\subsection{The Non-Geometric Case}\label{subsec:non-geometric}

In this section, we assume $f \from \Gamma \to \Gamma$ is a homotopy equivalence, $H \subset \Gamma$ is a $f$--invariant subgraph and $\rho = \pi_H^\perp([u,v]) \in C_1(\tGamma,\tH;\QQ)$ is a Nielsen 1--chain that is non-geometric.  As previously stated, the idea is to bound the norm of $A_{f,H}^k(x)$ in terms of the rational coefficients expressing $x$ as a linear combination of the translates of $\rho$ independent of $k$.  In this case, condition~\ref{nnc:single} provides the existence of an edge that is in the support for only a single translate of $\rho$, which makes the calculation straightforward.

\begin{lemma}\label{lem:nongeom estimate}
There is a constant $B \geq 1$ such that if $x = q_1g_1\rho + \cdots + q_r g_r \rho \in V_{\rm qf}$ and $k \geq 0$, then:
\begin{equation*}
\sum_{j=1}^r q_j^2 \leq \norm{A_{f,H}^k(x)}^2 \leq B\sum_{j=1}^r q_j^2.
\end{equation*}
\end{lemma}

\begin{proof}
Let $e_\rho \in \sfE(\tGamma) - \sfE(\tH)$ be an edge such that $\rho_{e_\rho} = \pm 1$ and $\rho_{ge_\rho} = 0$ for any non-trivial $g \in \FF$.  Such an edge exists by \ref{nnc:single}.   To simplify notation, we denote $A_{f,H}$ by $A$ in the proof.

As $A(\rho) = \rho$, for $k \geq 0$ we have $A^k(x) = q_1\Phi^k_f(g_1)\rho + \cdots + q_r\Phi^k_f(g_r)\rho$,
and hence it suffices to prove the lemma for $k = 0$ as the required bounds depend only on the rational coefficients and not the group elements determining the translates.

On one hand, observe that $x_{g_je_\rho} = \pm q_j$ since $(g_j\rho)_{g_je_\rho} = \pm  1$ and $(g\rho)_{g_je_\rho} = 0$ for all $g \in \FF$ not equal to $g_j$.  Thus:
\begin{equation*}
\sum_{j=1}^r q_j^2 = \sum_{j=1}^r x_{g_j e_\rho}^2 \leq \norm{x}^2.
\end{equation*}

On the other hand, letting $d = \#\abs{\supp(\rho)}$ we observe that for any edge $e \in \sfE(\tGamma) - \sfE(\tH)$, $(g\rho)_e = \rho_{g\inv e} \neq 0$ for at most $d$ elements $g \in \FF$.  Thus we find:
\begin{equation*}
x_e^2 = \left(\sum_{j=1}^r q_j(g_j\rho)_e\right)^2 \leq d^2\max\{q_j^2 \mid (g_j\rho)_e \neq 0 \}.
\end{equation*}  Additionally, each index $1 \leq j \leq r$ can realize the maximum value for at most $d$ edges as well.  Organizing the edges in $\supp(x)$ based on which index $j$ provides the maximal value on the given edge, we find:
\begin{equation*}
\norm{x}^2 = \sum_{e \in \supp(x)} x_e^2 \leq d^3\sum_{j=1}^r q_j^2.
\end{equation*}
Setting $B = d^3$ completes the proof. 
\end{proof}

With this estimate, we can prove Theorem~\ref{th:quasi fixed dynamics} when $V_{\rm qf}$ is generated by a non-geometric Nielsen 1--chain.

\begin{proposition}\label{prop:quasi fixed dynamics nongeom}
If $V_{\rm qf}$ is generated by a non-geometric Nielsen 1--chain, then there is a constant $C > 0$ such that for any $\xi \in V_{\rm qf}^{(2)}$ and $k \geq 0$ we have $C\inv\norm{\xi} \leq \norm{A_{f,H}^{k}(\xi)} \leq C\norm{\xi}$.
\end{proposition}

\begin{proof}
Let $B \geq 1$ be the constant from Lemma~\ref{lem:nongeom estimate} and set $C = \sqrt{B}$.  To simplify notation, we denote $A_{f,H}$ by $A$ in the proof.

We first prove the proposition for $x \in V_{\rm qf}$.  Write $x = q_1g_1\rho + \cdots + q_r g_r \rho$ for some rational numbers $q_1,\ldots,q_r \in \QQ$ and elements $g_1,\ldots,g_r \in \FF$.  By Lemma~\ref{lem:nongeom estimate} we find for any $k \geq 0$:
\begin{align*}
\norm{A^k(x)}^2 &\leq B\sum_{j=1}^r q_j^2 \leq B\norm{x}^2, \mbox{ and} \\
\norm{x}^2 &\leq B\sum_{j=1}^r q_j^2 \leq B\norm{A^k(x)}^2.
\end{align*}
This proves the proposition for $x \in V_{\rm qf}$.

Now given $\xi \in V^{(2)}_{\rm qf}$, $k \geq 0$ and $\epsilon > 0$, there exists $x_1, x_2 \in V_{\rm qf}$ such that both of 
\begin{equation*}
\abs{\norm{\xi}^2 - \norm{x_1 + ix_2}^2}, \, \mbox{and } \abs{\norm{A^k(\xi)}^2 - \norm{A^k(x_1 + ix_2)}^2}
\end{equation*}
are less than $\epsilon$.  As $A$ is a real operator, we have $\norm{A^k(x_1 + ix_2)}^2 = \norm{A^k(x_1)}^2 + \norm{A^k(x_2)}^2$.  Therefore:
\begin{align*}
\norm{A^k(\xi)}^2 & \leq \norm{A^k(x_1 + ix_2)}^2 + \epsilon \\
&= \norm{A^k(x_1)}^2 + \norm{A^k(x_2)}^2 + \epsilon \\
& \leq B\norm{x_1}^2 + B\norm{x_2}^2 + \epsilon \\
& \leq B\norm{\xi}^2 + \epsilon(B + 1).
\end{align*}
Similarly:
\begin{align*}
\norm{\xi}^2 & \leq \norm{x_1 + ix_2}^2 + \epsilon \\
&\leq \norm{x_1}^2 + \norm{x_2}^2 + \epsilon \\
& \leq B\norm{A^k(x_1)}^2 + B\norm{A^k(x_2)}^2 + \epsilon \\
& \leq B\norm{A^k(\xi)}^2 + \epsilon(B + 1).
\end{align*}
As this holds for all $\epsilon > 0$, we have $C\inv\norm{\xi} \leq \norm{A^k(\xi)} \leq C\norm{\xi}$ as desired.
\end{proof}


\subsection{The Geometric Case}\label{subsec:geometric}

In this section, we assume $f \from \Gamma \to \Gamma$ is a homotopy equivalence, $H \subset \Gamma$ is a $f$--invariant subgraph and $\rho = \pi_H^\perp([u,v]) \in C_1(\tGamma,\tH;\QQ)$ is a Nielsen 1--chain that is geometric and that generates $V_{\rm qf}$.  Again, as previously stated, the key idea is to bound the norm of $A_{f,H}^k(x)$ in terms of the rational coefficients expressing $x$ as a linear combination of the translates of $\rho$ independent of $k$.  In this case, we will work with an auxiliary graph $T_\rho$ that captures the combinatorics of the translates of $\rho$.  The graph $T_\rho$ has a free action by $\FF$ and we consider the $\QQ[\FF]$--modules of compactly supported 0-- and 1--cochains $C_c^0(T_\rho;\QQ)$ and $C_c^1(T_\rho;\QQ)$ respectively.  Of importance is the coboundary operator $\delta_0 \from C_c^0(T_\rho;\QQ) \to C_c^1(T_\rho;\QQ)$ which is a $\QQ[\FF]$--module homomorphism given by:
\begin{equation*}
\delta_0\psi(\varepsilon) = \psi(\sft(\varepsilon)) - \psi(\sfo(\varepsilon)).
\end{equation*} 
where $\varepsilon \in \sfE(T_\rho)$.  We will consider the usual $L^2$--norms on both $C_c^0(T_\rho;\QQ)$ and $C_c^1(T_\rho;\QQ)$.  As our cochains are compactly support, these norms is well-defined.  

We will define a ``realization'' map $R \from V_{\rm qf} \to C_c^0(T_\rho;\QQ)$ and show that the sum of the squares of the rational coefficients of $x$ equals $\norm{R(x)}^2$ and also that $\norm{x}$ equals $\norm{\delta_0R(x)}$.  This takes place in Lemma~\ref{lem:realize qausi-fix}.  There is bi-Lipschitz relation between $\norm{\psi}$ and $\norm{\delta_0\psi}$ that we recall in Lemma~\ref{lem:coboundary estimate}.  This gives us the desired relation between  the sum of the squares of the rational coefficients of $x$ and $\norm{x}^2$ from which Proposition~\ref{prop:quasi fixed dynamics geom} follows in a similar way to Proposition~\ref{prop:quasi fixed dynamics nongeom}.

The graph $T_\rho$ is defined by the following data.
\begin{align*}
\sfV(T_\rho) &= \FF \\
\sfE(T_\rho) &= \{ [g_1,g_2] \mid \supp(g_1\rho) \cap \supp(g_2\rho) \neq \emptyset   \}
\end{align*}
The graph $T_\rho$ is not connected in general, but if $T_0$ and $T_1$ are components of $T_\rho$, then there is an element $g \in \FF$ such that $T_1 = gT_0$.  We note that there is a bijection between the edges of $T_\rho$ and the edges $e \in \sfE(\tGamma) - \sfE(\tH)$.  Indeed, \ref{gnc:single} implies the assignment that sends an edge $[g_1,g_2] \in \sfE(T_\rho)$ to $\supp(g_1\rho) \cap \supp(g_2\rho)$ defines a function from $\sfE(T_\rho) \to \sfE(\tGamma) - \sfE(\tH)$ and \ref{gnc:pair} implies this function is a bijection.

\begin{example}\label{ex:t rho}
Let $\Gamma$ be the theta graph labeled as in Figure~\ref{fig:theta}.  In this example, the homotopy equivalence $f \from \Gamma \to \Gamma$ and subgraph $H \subset \Gamma$ are irrelevant and will not be specified.     

\begin{figure}[ht]
\centering
\begin{tikzpicture}
\draw[very thick] (0,0) circle [x radius=1.8, y radius=1];
\draw[very thick] (-1.8,0) -- (1.8,0);
\fill (-1.8,0) circle [radius=0.075];
\fill (1.8,0) circle [radius=0.075];
\draw[thick] (0.05,1.1) -- (-0.05,1) -- (0.05,0.9);
\draw[thick] (0.05,-0.1) -- (-0.05,0) -- (0.05,0.1);
\draw[thick] (0.05,-1.1) -- (-0.05,-1) -- (0.05,-0.9);
\node at (0,1.3) {\footnotesize $a$};
\node at (0,0.3) {\footnotesize $b$};
\node at (0,-0.7) {\footnotesize $c$};
\node at (2.05,0) {\footnotesize $\ast$};
\end{tikzpicture}
\caption{The graph $\Gamma$ in Example~\ref{ex:t rho}.}\label{fig:theta}
\end{figure}
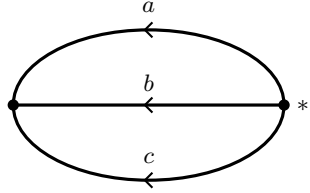

There is an isomorphism $\pi_{1}(\Gamma,*) \cong \FF = \I{x_{1},x_{2}}$ where $x_{1}$ corresponds to the edge-path $a\bar{b}$ and $x_{2}$ corresponds to th edge-path $c\bar{a}$.  Fix a lift $\tast$ of $\ast$ to $\tGamma$ that lies on the axes of $x_1$ and $x_2$.  We consider the edge-path from $\tast$ to $x_1x_2x_1\inv x_2\inv \tast$.  We fix lifts of $a$, $b$ and $c$ respectively in $\tGamma$ as pictured in Figure~\ref{fig:t rho} and abusing notation continue to denote them by $a$, $b$ and $c$ respectively.  Then the 1--chain:
\begin{equation*}
\rho = [\tast,x_1x_2x_1\inv x_2\inv \tast] = a - b + x_1c - x_1x_2 a + x_1x_2x_1\inv b - x_1x_2x_1\inv x_2\inv c  
\end{equation*}  
satisfies conditions \ref{gnc:single}, \ref{gnc:pair} and \ref{gnc:noncommuting}.  Indeed, \ref{gnc:single} is apparent since the edge-path $a\bar{b}c\bar{a}b\bar{c}$ in $\Gamma$ does not repeat a two-letter subword, \ref{gnc:pair} is apparent as $\supp(\rho)$ contains exactly two edges from each orbit of edges in $\tGamma$ and \ref{gnc:noncommuting} holds for $g_1 = x_1x_2$ and $g_2 = x_1x_2x_1\inv$.  The six translates of $\rho$ whose support has non-empty intersection with the support of $\rho$ are also illustrated in Figure~\ref{fig:rho}.  The corresponding portion of $T_\rho$ is illustrated in Figure~\ref{fig:t rho}.  
   
\begin{figure}[ht]
\centering
\begin{tikzpicture}
\def\s{1.1}
\draw[very thick,blue] (0,0) -- (12*\s,0);
\fill[blue] (0,0) circle [radius=0.075];
\fill[blue] (12*\s,0) circle [radius=0.075];
\foreach \a in {1,...,5} {
	\fill (0 + \a*2*\s,0) circle [radius=0.075];
}
\foreach \a in {0,2,4} {
	\draw[thick,blue] (\s-0.1 + 2*\a*\s,-0.1) -- (\s + 2*\a*\s,0) -- (\s-0.1 + 2*\a*\s,0.1);
	\draw[thick,blue] (3*\s+0.1 + 2*\a*\s,-0.1) -- (3*\s + 2*\a*\s,0) -- (3*\s+0.1 + 2*\a*\s,0.1);
}
\draw[very thick] (0,0) -- (0,-2);
\draw[thick] (-0.1,-1.1) -- (0,-1) -- (0.1,-1.1);
\draw[very thick] (2*\s,2) -- (2*\s,0) (6*\s,-2) -- (6*\s,0) (10*\s,2) -- (10*\s,0);
\draw[thick] (2*\s-0.1,1.1) -- (2*\s,1) -- (2*\s+0.1,1.1) (6*\s-0.1,-1.1) -- (6*\s,-1) -- (6*\s+0.1,-1.1) (10*\s-0.1,1.1) -- (10*\s,1) -- (10*\s+0.1,1.1);
\fill (0,-2) circle [radius=0.075];
\fill (2*\s,2) circle [radius=0.075];
\fill (6*\s,-2) circle [radius=0.075];
\fill (10*\s,2) circle [radius=0.075];
\draw[very thick, red] (1.3*\s,3) -- (2*\s-0.2,2) -- (2*\s-0.2,0.25) -- (0.2,0.25) -- (0.2,4);
\draw[thick,red] (0.1,1.4) -- (0.2,1.5) -- (0.3,1.4);
\fill[red] (1.3*\s,3) circle [radius=0.075];
\fill[red] (0.2,4) circle [radius=0.075];
\draw[very thick, red] (4*\s-0.2,4) -- (4*\s-0.2,0.25) -- (2*\s+0.2,0.25) -- (2*\s+0.2,2);
\draw[thick,red] (4*\s-0.3,1.6) -- (4*\s-0.2,1.5) -- (4*\s-0.1,1.6);
\fill[red] (4*\s-0.2,4) circle [radius=0.075];
\fill[red] (2*\s+0.2,2) circle [radius=0.075];
\draw[very thick, red] (4*\s+0.2,-4) -- (6*\s-0.2,-2) -- (6*\s-0.2,-0.25) -- (4*\s,-0.25); 
\draw[thick,red] (5*\s-0.1414,-3) -- (5*\s,-3) -- (5*\s,-3.1414);
\fill[red] (4*\s+0.2,-4) circle [radius=0.075];
\fill[red] (4*\s,-0.25) circle [radius=0.075];
\draw[very thick, red] (8*\s,-0.25) -- (6*\s+0.2,-0.25) -- (6*\s+0.2,-2) -- (8*\s-0.2,-4); 
\draw[thick,red] (7*\s-0.1414,-3) -- (7*\s,-3) -- (7*\s,-2.8586);
\fill[red] (8*\s,-.25) circle [radius=0.075];
\fill[red] (8*\s-0.2,-4) circle [radius=0.075];
\draw[very thick, red] (10*\s-0.2,2) -- (10*\s-0.2,0.25) -- (8*\s+0.2,0.25) -- (8*\s+0.2,4);
\draw[thick,red] (8*\s+0.1,1.4) -- (8*\s+0.2,1.5) -- (8*\s+0.3,1.4);
\fill[red] (10*\s-0.2,2) circle [radius=0.075];
\fill[red] (8*\s+0.2,4) circle [radius=0.075];
\draw[very thick, red] (12*\s-0.2,4) -- (12*\s-0.2,0.25) -- (10*\s+0.2,0.25) -- (10*\s+0.2,2) -- (10.7*\s,3); 
\draw[thick,red] (12*\s-0.1,1.6) -- (12*\s-0.2,1.5) -- (12*\s-0.3,1.6);
\fill[red] (12*\s-0.2,4) circle [radius=0.075];
\fill[red] (10.7*\s,3) circle [radius=0.075];
\node at (-0.5,0) {\footnotesize $\rho\from \tast$};
\node at (12*\s+1.4,0) {\footnotesize $x_1x_2x_1\inv x_2\inv\tast$};
\node at (\s,-0.4) {\footnotesize $a$};
\node at (3*\s,-0.4) {\footnotesize $b$};
\node at (0.3,-1.1) {\footnotesize $c$};
\node at (5*\s,0.4) {\footnotesize $x_1c$};
\node at (7*\s,0.4) {\footnotesize $x_1x_2a$};
\node at (9*\s,-0.4) {\footnotesize $x_1x_2x_1\inv b$};
\node at (11*\s,-0.4) {\footnotesize $x_1x_2x_1\inv x_2\inv c$};
\node at (2*\s+0.7,1) {\footnotesize $x_2\inv c$};
\node at (6*\s-0.75,-1) {\footnotesize $x_1x_2b$};
\node at (10*\s-1,1) {\footnotesize $x_1x_2x_1\inv a$};
\node at (1.3*\s,3.4) {\footnotesize $x_2\inv x_1\inv \rho$};
\node at (4*\s-0.2,4.4) {\footnotesize $x_1x_2\inv x_1\inv \rho$};
\node at (4*\s-1.3,-4) {\footnotesize $x_1x_2x_1x_2\inv x_1\inv \rho$};
\node at (8*\s-0.2,-0.7) {\footnotesize $x_1x_2\rho$};
\node at (10*\s-0.75,2.4) {\footnotesize $x_1x_2 x_1\inv \rho$};
\node at (12*\s-0.2,4.4) {\footnotesize $x_1x_2x_1\inv x_2\inv x_1\inv\rho$};
\end{tikzpicture}
\caption{The Nielsen 1--chain $\rho$ in Example~\ref{ex:t rho} and the six translates whose support intersects the support of $\rho$.}\label{fig:rho}
\end{figure}
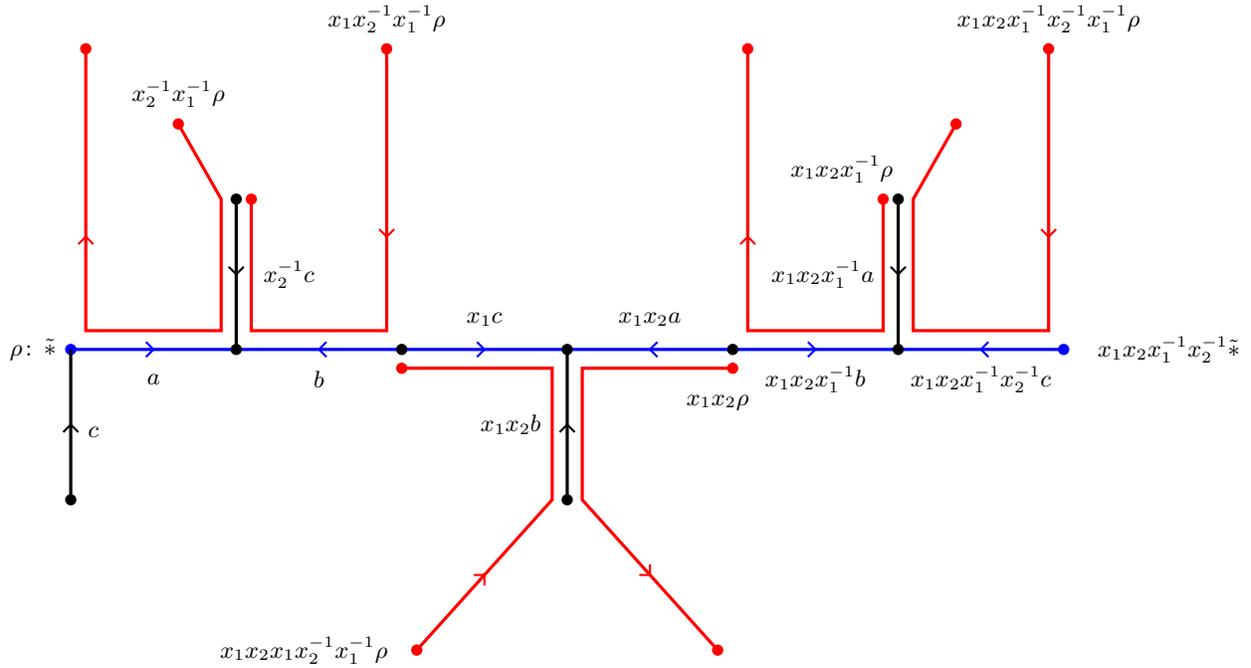

\begin{figure}[ht]
\centering
\begin{tikzpicture}
\node[regular polygon, regular polygon sides=6,minimum size=5cm](P) at (0,0) {};
\fill (0,0) circle [radius=0.075];
\foreach \a in {1,...,6} {
	\fill (P.corner \a) circle [radius=0.075];
}
\foreach \a in {1,3,5} {
	\pgfmathsetmacro\b{int(\a + 1)};
	\draw[very thick] (P.corner \a) -- (P.corner \b) -- (0,0) -- cycle;
}
\node at (0,-0.5) {\footnotesize $1$};
\node[above] at (P.corner 2) {\footnotesize $x_2\inv x_1\inv$};
\node[above] at (P.corner 1) {\footnotesize $x_1x_2\inv x_1\inv$};
\node[right] at (P.corner 6) {\footnotesize $x_1x_2x_1x_2\inv x_1\inv$};
\node[below] at (P.corner 5) {\footnotesize $\phantom{\inv}x_1x_2\phantom{\inv}$};
\node[below] at (P.corner 4) {\footnotesize $x_1x_2x_1\inv$};
\node[left] at (P.corner 3) {\footnotesize $x_1x_2x_1\inv x_2\inv x_1\inv$};
\end{tikzpicture}
\caption{A portion of the graph $T_\rho$ in Example~\ref{ex:t rho}.}\label{fig:t rho}
\end{figure}
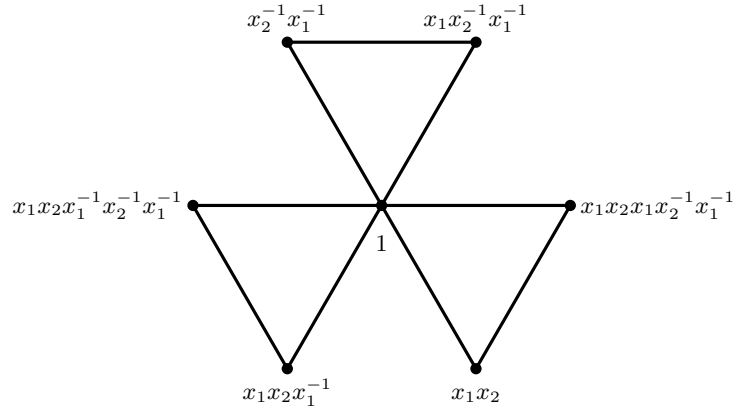
\end{example}

There is an obvious map $R' \from V_{\rm qf} \to C_c^0(T_\rho;\QQ)$ that sends a 1--chain $x = q_1g_1\rho + \cdots + q_rg_r\rho \in V_{\rm qf}$ to the function $R'(x) = q_1\chi_{g_1} + \cdots + q_r\chi_{g_r}$ where $\chi_g$ is the characteristic function of the set $\{g\} \subset \sfV(T_\rho)$. It is clearly true that $\sum_{j=1}^r q_j^2 = \norm{R'(x)}^2$.  Additionally, for the Nielsen 1--chain as in Example~\ref{ex:t rho}, we can demonstrate that $\norm{x} = \norm{\delta_0R'(x)}$.  Indeed, if $e$ is the unique edge in $\supp(g\rho) \cap \supp(g'\rho)$, then, up to sign, the coefficient $x_e$ is equal to $\sum_{j=1}^r q_j(\chi_{g'}(g_j) - \chi_{g}(g_j))$.  (Notice that there are at most two nonzero terms in the sum.)  This follows as for the Nielsen 1--chain in Example~\ref{ex:t rho}, for any edge, the coefficients of its translates in $\supp(\rho)$ are $1$ and $-1$.  Next, if $\varepsilon \in \sfE(T_\rho)$ is the edge corresponding to $\supp(g\rho) \cap \supp(g'\rho)$, then up to swapping the orientation of $\varepsilon$ we have $\sfo(\varepsilon) = g$ and $\sft(\varepsilon) = g'$, hence $\delta_0R'(x)(\varepsilon)$ is equal to $\sum_{j=1}^r q_j(\chi_{g_j}(g') - \chi_{g_j}(g))$.  (Again, there are at most two nonzero terms.) This shows that $\norm{\delta_0R'(x)} = \norm{x}$ as claimed since $\chi_{g'}(g) = \chi_g(g')$ for any $g,g' \in \FF$.

In general though, it is not the case that the coefficients of the translates of any edge in $\supp(\rho)$ are $1$ and $-1$ and so we need to take such edges into account when defining the realization map $R \from V_{\rm qf} \to C_c^0(T_\rho;\QQ)$.  We call an edge $e \in \sfE(\tGamma) - \sfE(\tH)$ \emph{non-orientable} if the coefficients of its translates in $\supp(\rho)$ have the same sign, in which case they are either are both $1$ or either both $-1$.  Likewise, we call an edge $\varepsilon \in \sfE(T_\rho)$ non-orientable is the corresponding edge in $\sfE(\tGamma) - \sfE(\tH)$ is non-orientable.  In this language, there are no non-orientable edges in Example~\ref{ex:t rho}.  

If the unique edge in $\supp(g_j\rho) \cap \supp(g_k\rho)$ is non-orientable, to get an equality $\norm{x} = \norm{\delta_0R(x)}$, we need to make sure that we have $(R(x)(g_k) - R(x)(g_j))^2 = (q_k + q_j)^2$ and so one of the terms in the expression for $R(x)$ needs to be multiplied by $-1$.  One way to accomplish this is the following.  Fix a vertex $g_0 \in \sfV(T_\rho)$, let $m$ denote the number of non-orientable edges in an edge-path from $g_0$ to the vertex $g_j \in \sfV(T_\rho)$ and use $(-1)^mq_j\chi_{g_j}$ in the summation formula for $R(x)$.  If $T_\rho$ is a tree, then this construction is obviously well-defined.  However $T_\rho$ is not a tree in general.  Nonetheless, we can still show that the parity of the number of non-orientable edges in an edge path is well-defined.  

To this end, suppose $\bp \from p_0,\ldots,p_m$ is an edge-path in $T_\rho$ where $p_0 = p_m$.  Recall from Section~\ref{subsec:graphs morphisms} that notation this means the edge-path $\bp$ visits the vertices $p_0,\ldots,p_m$ in this order.  This unambiguously defines an edge-path since $T_\rho$ is a simplicial graph.  To simplify the exposition, the index $j$ in the remainder of this description is considered modulo $m$.  Let $g_0,\ldots,g_{m-1} \in \FF$ be the elements such that $p_j = g_j$ for $0 \leq j < m$.  For each $0 \leq j < m$, the intersection $\supp(g_{j-1}\rho) \cap \supp(g_j\rho)$ consists of a single edge $e_j \in \sfE(\tGamma) - \sfE(\tH)$.  Thus we have $e_j,e_{j+1} \in \supp(g_j\rho)$ for $0 \leq j < m$.  An \emph{orientation} on $\bp$ is a choice of orientations on the edges $e_j$ so that $(g_j\rho)_{e_j} = -(g_j\rho)_{e_{j+1}}$ for $0 \leq j < m$.  Note that the orientations on the edges are not assumed to be $\FF$--equivariant.  

\begin{lemma}\label{lem:orientable}
Suppose $\bp \from p_0,\ldots,p_m$ is an edge-path in $T_\rho$ where $p_0 = p_m$.  Then either $p_j = p_k$ for some distinct indices $0 \leq j,k < m$ or there exists an orientation on $\bp$. 
\end{lemma}

\begin{proof}
To simplify the exposition, subscript indices are considered modulo $m$ in this proof.

There are elements $g_0,\ldots,g_{m-1} \in \FF$ such that $p_j = g_j$ for $0 \leq j < m$.  For each $0 \leq j < m$, the intersection $\supp(g_{j-1}\rho) \cap \supp(g_j\rho)$ consists of a single edge $e_j \in \sfE(\tGamma) - \sfE(\tH)$. We have $e_j,e_{j+1} \in \supp(g_j\rho)$ for $0 \leq j < m$.  If $e_j = e_{k}$ for some distinct indices $0 \leq j, k < m$, then by \ref{gnc:pair}, the pair of elements $\{g_{j-1},g_j\}$ equals the pair of elements $\{g_{k-1},g_k\}$.  If $k \equiv j+1 \pmod{m}$, then the lemma holds as either $p_k = p_{j}$ or $p_{k} = p_{j-1}$.  Else, we see that the lemma holds as either $p_j = p_{k-1}$ or $p_j = p_k$.  

Therefore, we can assume that the set of edges $e_0,\ldots,e_m$ are distinct for the remainder of the proof.  

For each $0 \leq j < m$, we have that the pair of edges $e_j,e_{j+1}$ belong to the edge-path $g_j[u,v] \subset \tGamma$.  We locally orient $e_j$ and $e_{j+1}$ to point away from each other.  In other words, there is a component of $\tGamma - \{\sft(e_j),\sft(e_{j+1})\}$ that contains both $\sfo(e_j)$ and $\sfo(e_{j+1})$.

We claim that if these local choices are not consistent on $\bp$ then $p_j = p_k$ for some distinct indices $0 \leq j,k < m$.  Clearly, if these choices are consistent then $(g_j\rho)_{e_j} = -(g_j\rho)_{e_{j+1}}$ for all $0 \leq j <m$ and so they determine an orientation on $\bp$. 

If these local choices are not consistent, then there is some index $0 \leq j < m$ such that the local orientation on $e_j$ induced from $g_{j-1}[u,v]$ does not equal the local orientation on $e_j$ induced from $g_j[u,v]$.   This implies that the edge $e_j$ lies on the edge path from $e_{j-1}$ to $e_{j+1}$.  Indeed, let $w$ be midpoint of the edge $e_j$, let $w_{j-1}$ denote the initial vertex of $e_j$ in the orientation induced from $g_{j-1}[u,v]$ and let $w_j$ denote the initial vertex of $e_j$ in the orientation induced from $g_j[u,v]$.  Then the edge-path from $e_{j-1}$ to $w$ goes through $w_{j-1}$ and the edge-path from $e_{j+1}$ to $w$ goes through $w_j$.  As $w_{j-1}$ and $w_j$ are distinct, this shows that the edge-path from $e_{j-1}$ to $e_{j+1}$ contains $w$ and hence also $e_j$ as claimed.

Let $\bq$ be the edge path in $\tGamma$ from $g_{j-2}[u,v]$ to $g_{j+1}[u,v]$.  This edge path contains $e_j$.  This follows from \ref{gnc:single} as $e_{j-1} \subset g_{j-2}[u,v]$, $e_{j+1} \subset g_{j+1}[u,v]$ and $e_j$ is contained in neither $g_{j-2}[u,v]$ nor $g_{j+1}[u,v]$.  See Figure~\ref{fig:orientation}.  

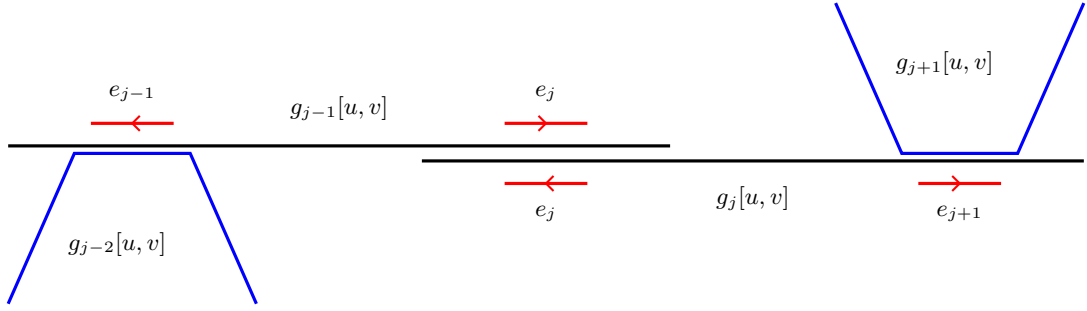
\begin{figure}[ht]
\centering
\begin{tikzpicture}
\def\s{1.1}
\draw[very thick] (0,0.1) -- (8*\s,0.1) (5*\s,-0.1) -- (13*\s,-0.1);
\draw[very thick, red] (1*\s,0.4) -- (2*\s,0.4) (6*\s,0.4) -- (7*\s,0.4) (6*\s,-0.4) -- (7*\s,-0.4) (11*\s,-0.4) -- (12*\s,-0.4);
\draw[thick,red] (1.5*\s+0.1,0.5) -- (1.5*\s,0.4) -- (1.5*\s+0.1,0.3);
\draw[thick,red] (6.5*\s-0.1,0.5) -- (6.5*\s,0.4) -- (6.5*\s-0.1,0.3);
\draw[thick,red] (6.5*\s+0.1,-0.5) -- (6.5*\s,-0.4) -- (6.5*\s+0.1,-0.3);
\draw[thick,red] (11.5*\s-0.1,-0.5) -- (11.5*\s,-0.4) -- (11.5*\s-0.1,-0.3);
\draw[very thick, blue] (0,-2) -- (0.8*\s,0) -- (2.2*\s,0) -- (3*\s,-2);
\draw[very thick, blue] (10*\s,2) -- (10.8*\s,0) -- (12.2*\s,0) -- (13*\s,2);
\node at (1.5*\s,0.8) {\footnotesize $e_{j-1}$};
\node at (6.5*\s,0.8) {\footnotesize $e_{j}$};
\node at (6.5*\s,-0.8) {\footnotesize $e_{j}$};
\node at (11.5*\s,-0.8) {\footnotesize $e_{j+1}$};
\node at (1.5*\s-0.2,-1.2) {\footnotesize $g_{j-2}[u,v]$};
\node at (4*\s,0.6) {\footnotesize $g_{j-1}[u,v]$};
\node at (9*\s,-0.6) {\footnotesize $g_{j}[u,v]$};
\node at (11.5*\s-0.2,1.2) {\footnotesize $g_{j+1}[u,v]$};
\end{tikzpicture}
\caption{Inconsistent local orientations on the edge $e_j$ in the proof of Lemma~\ref{lem:orientable}.}\label{fig:orientation}
\end{figure}

The union $\bigcup_{r = j+1}^{j-2} g_r[u,v]$ is connected as $g_{r}[u,v] \cap g_{r+1}[u,v]$ is nonempty for all $0 \leq r < m$.  Hence, this union also contains the edge-path $\bq$ and thus the edge $e_j$.  This implies that $e_j \in \supp(g_k\rho)$ for some $k \neq j,j+1$.  By \ref{gnc:pair}, we must have that $g_k = g_j$ or $g_k = g_{j+1}$.  This shows that two vertices of $\bp$ are the same and completes the proof.  
\end{proof}

Motivated by the previous discussion, we introduce the following notion.  Let $\bp\from p_0, \ldots, p_m$ be an edge-path in $T_\rho$ and let $g_0,\ldots,g_m \in \FF$ be the elements such that $p_j = g_j$ for $0 \leq j \leq m$.  For each $1 \leq j \leq m$, the intersection $\supp(g_{j-1}\rho) \cap \supp(g_j\rho)$ consists of a single edge $e_j \in \sfE(\tGamma) - \sfE(\tH)$.  We have $(g_{j-1}\rho)_{e_j}, (g_j\rho)_{e_j} \in \{-1, 1\}$.  We set $\sigma(\bp)$ to be equal to the number of indices $1 \leq j \leq m$ such that $(g_{j-1}\rho)_{e_j} = (g_j\rho)_{e_j}$.  In other words, $\sigma(\bp)$ is the number of non-orientable edges along the edge-path $\bp$.  In particular, we have:
\begin{equation}\label{eq:sigma product formula}
(-1)^{\sigma(\bp)} = \prod_{j=1}^m \left(-\frac{(g_{j-1}\rho)_{e_j}}{(g_j\rho)_{e_j}}\right).
\end{equation}  
We remark that $\sigma(\bp)$ is well-defined independent of choice of orientation on the edges $e_j$. By definition, if $\bp$ is a trivial path, then $\sigma(\bp) = 0$ so that $(-1)^{\sigma(\bp)} = 1$ and \eqref{eq:sigma product formula} holds where we define the empty product to be equal to $1$.

We seek to show that $(-1)^{\sigma(\bp)}$ only depends on the endpoints of the edge-path $\bp$.  The next lemma shows this is true for orientable circuits.  We extend this to all circuits in Lemma~\ref{lem:sigma well-defined} and to all edge-paths in Corollary~\ref{co:sigma well-defined}.

\begin{lemma}\label{lem:sigma on orientable circuit}
Suppose $\bp \from p_0,\ldots,p_m$ is an edge-path in $T_\rho$ where $p_0 = p_m$.  If $\bp$ is orientable, then $(-1)^{\sigma(\bp)} = 1$.
\end{lemma}
      
\begin{proof}
To simplify the exposition, subscript indices are considered modulo $m$ in this proof.

Let $g_0,\ldots,g_{m-1} \in \FF$ be the elements so that $p_j = g_j$ for $0 \leq j < m$.  For each $0 \leq j < m$, there is an edge $e_j$ that is the unique edge in $\supp(g_{j-1}\rho) \cap \supp(g_j\rho)$.  As $\bp$ is orientable, we may choose orientations on the edges $e_j$ such that $(g_{j}\rho)_{e_j} = -(g_j\rho)_{e_{j+1}}$ for $0 \leq j < m$.  

Thus, using~\eqref{eq:sigma product formula}, we find:
\begin{equation*}
(-1)^{\sigma(\bp)}  = \prod_{j=1}^{m}\left(-\frac{(g_{j-1}\rho)_{e_j}}{(g_j\rho)_{e_j}}\right)
= \prod_{j=0}^{m-1} \left(-\frac{(g_{j}\rho)_{e_{j+1}}}{(g_j\rho)_{e_j}}\right) = 1.\qedhere
\end{equation*}
\end{proof}

\begin{lemma}\label{lem:sigma well-defined}
Suppose that $\bp\from p_0, \ldots, p_{m}$ is an edge-path in $T_\rho$ where $p_0 = p_{m}$.  Then $(-1)^{\sigma(\bp)} = 1$.
\end{lemma}

\begin{proof}
Assume the statement of the lemma is false.  Fix an edge-path $\bp\from p_0, \ldots, p_{m}$ where $p_0 = p_{m}$ and $(-1)^{\sigma(\bp)} = -1$ with $m$ minimal among all such edge-paths $\bp$.  To simplify the exposition, subscript indices are considered modulo $m$ in this proof.
    
By Lemma~\ref{lem:sigma on orientable circuit}, the path $\bp$ is not orientable.  Hence, by Lemma~\ref{lem:orientable}, there are distinct indices $0 \leq j,k < m$ such that $p_j = p_k$.  We consider the following edge-paths: $\bp' \from p_j,p_{j+1},\ldots,p_k$ and $\bp'' \from p_k,p_{k+1},\ldots,p_j$.  Notice that $\sigma(\bp) = \sigma(\bp') + \sigma(\bp'')$.  Hence either $(-1)^{\sigma(\bp')} = -1$ or $(-1)^{\sigma(\bp'')} = -1$.  However, this is a contradiction as $p_j = p_k$ and the lengths of $\bp'$ and $\bp''$ are less than $m$.

This contradiction proves the lemma.
\end{proof}

From this, it follows readily that $(-1)^{\sigma(\bp)}$ depends only on the endpoints of $\bp$ as desired.

\begin{corollary}\label{co:sigma well-defined}
Suppose that $\bp\from p_0, \ldots, p_m$ and $\bq\from q_0, \ldots,q_n$ are edge-paths in $T_\rho$ where $p_0 = q_0$ and $p_m = q_n$.  Then $(-1)^{\sigma(\bp)} = (-1)^{\sigma(\bq)}$.
\end{corollary}

\begin{proof}
Apply Lemma~\ref{lem:sigma well-defined} to the edge-path $\bp\overline{\bq}$ and note that $\sigma(\bp\overline{\bq}) = \sigma(\bp) + \sigma(\bq)$.
\end{proof}

In each component $T_0 \subset T_\rho$ we fix a vertex $g_0$.  For $g \in \sfV(T_0)$, we fix some path $\bp_g$ from $g_0$ to $g$ and define $\sign(g) = (-1)^{\sigma(\bp_g)}$.  By Corollary~\ref{co:sigma well-defined}, the function $\sign\from \FF \to \{-1,1\}$ is well-defined. We can now define the realization map $R\from V_{\rm qf} \to C_c^0(T_\rho;\QQ)$.  Given $x = q_1g_1\rho + \cdots + q_rg_r\rho$, we set: 
\begin{equation*}
R(x) = \sign(g_1)q_1\chi_{g_1} + \cdots + \sign(g_r)q_r\chi_{g_r}.
\end{equation*}

\begin{example}\label{ex:realization}
Let $\Gamma$ be the 3--rose labeled as in Figure~\ref{fig:rose}.  As in Example~\ref{ex:t rho}, the homotopy equivalence $f\from \Gamma \to \Gamma$ and subgraph $H \subset \Gamma$ are irrelevant and will not be specified.

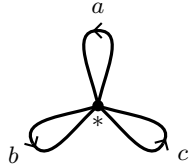
\begin{figure}[ht]
\centering
\begin{tikzpicture}
\begin{scope}[xshift=0cm,rotate=0]
\draw[very thick] (0,0) to[out=80,in=0] (0,1);
\draw[very thick] (0,1) to[out=180,in=100] (0,0);
\draw[thick] (0.05,0.9) -- (-0.05,1) -- (0.05,1.1);
\node at (0,1.3) {\footnotesize $a$};
\end{scope}
\begin{scope}[xshift=0cm,rotate=120]
\draw[very thick] (0,0) to[out=80,in=0] (0,1);
\draw[very thick] (0,1) to[out=180,in=100] (0,0);
\draw[thick] (0.05,0.9) -- (-0.05,1) -- (0.05,1.1);
\node at (0,1.3) {\footnotesize $b$};
\end{scope}
\begin{scope}[xshift=0cm,rotate=240]
\draw[very thick] (0,0) to[out=80,in=0] (0,1);
\draw[very thick] (0,1) to[out=180,in=100] (0,0);
\draw[thick] (0.05,0.9) -- (-0.05,1) -- (0.05,1.1);
\node at (0,1.3) {\footnotesize $c$};
\end{scope}
\fill (0,0) circle [radius=0.075];
\node[below] at (0,0) {\footnotesize $\ast$};
\end{tikzpicture}
\caption{The graph $\Gamma$ in Example~\ref{ex:realization}.}\label{fig:rose}
\end{figure}

There is an isomorphism $\pi_{1}(\Gamma,*) \cong \FF = \I{x_{1},x_{2},x_3}$ where $x_{1}$ corresponds to the edge-path $a$ and $x_{2}$ corresponds to the edge-path $b$ and $x_3$ corresponds to the edge-path $c$.  Fix a lift $\tast$ of $\ast$ to $\tGamma$ that lies on the axes of $x_1$, $x_2$ and $x_3$.  We consider the edge-path from $\tast$ to $x_1^2x_2x_3x_2\inv x_3\inv \tast$.  We fix lifts of $a$, $b$ and $c$ respectively in $\tGamma$ as pictured in Figure~\ref{fig:realization} and abusing notation continue to denote them by $a$, $b$ and $c$ respectively.  Then, in a similar way as in Example~\ref{ex:t rho}, we can see that the 1--chain:
\begin{equation*}
\rho = [\tast,x_1^2x_2x_3x_2\inv x_3\inv \tast] = a + x_1a + x_1^2 b + x_1^2 x_2 c - x_1^2x_2x_3x_2\inv b - x_1^2x_2x_3x_2\inv x_3\inv c  
\end{equation*}  
satisfies conditions \ref{gnc:single}, \ref{gnc:pair} and \ref{gnc:noncommuting}.  Let $g_1 = x_1$, $g_2 = x_1^3x_2x_3x_2\inv x_1^{-2}$ and $g_3 = x_1^2x_2x_3x_2\inv x_1^{-2}$.  Let $x = q_0\rho + q_1g_1\rho + q_2g_2\rho + q_3g_3\rho$.  The 1--chain $\rho$ along with these three translates are pictured in Figure~\ref{fig:realization}.  Using the vertex corresponding to the identity in $T_\rho$, we see that $\sign(g_1) = -1$, $\sign(g_2)= -1$ and $\sign(g_3) = 1$ and thus we have:
\begin{equation*}
R(x) = q_0\chi_{\id_\FF} - q_1\chi_{g_1} - q_2\chi_{g_2} + q_3\chi_{g_3}.
\end{equation*} 
For this $\rho$, the graph $T_\rho$ is the 6--regular tree.  The vertices and edges respectively with non-zero values for the 0--cochain $R(x)$ and its coboundary $\delta_0R(x)$ respectively are shown in Figure~\ref{fig:cochain}.  From this is it readily verified that $\norm{x} = \norm{\delta_0R(x)}$.

\begin{figure}[ht]
\centering
\begin{tikzpicture}
\def\s{1.1}
\draw[very thick] (0,2) -- (0,-2);
\draw[very thick,blue] (0,0) -- (12*\s,0);
\fill[blue] (0,0) circle [radius=0.075];
\fill[blue] (12*\s,0) circle [radius=0.075];
\foreach \a in {1,...,5} {
	\fill (0 + \a*2*\s,0) circle [radius=0.075];
}
\foreach \a in {0,1,2,3} {
	\draw[thick,blue] (\s-0.1 + 2*\a*\s,-0.1) -- (\s + 2*\a*\s,0) -- (\s-0.1 + 2*\a*\s,0.1);
}
\foreach \a in {4,5} {
	\draw[thick,blue] (\s+0.1 + 2*\a*\s,-0.1) -- (\s + 2*\a*\s,0) -- (\s+0.1 + 2*\a*\s,0.1);
}
\draw[thick] (-0.1,-0.9) -- (0,-1) -- (0.1,-0.9);
\draw[thick] (-0.1,0.9) -- (0,1) -- (0.1,0.9);
\draw[very thick,red] (2*\s,0.2) -- (4*\s-0.2,0.2) -- (4*\s-0.2,4);
\draw[thick,red] (4*\s-0.3,1.4) -- (4*\s-0.2,1.5) -- (4*\s-0.1,1.4);
\fill[red] (2*\s,0.2) circle [radius=0.075];
\fill[red] (4*\s-0.2,4) circle [radius=0.075];
\draw[very thick, red] (11*\s-0.2,-2.2) -- (10*\s-0.2,-0.2) -- (8*\s+0.2,-0.2) -- (7*\s+0.2,-2.2);
\begin{scope}[xshift=11.35 cm,yshift=-1.2 cm,rotate=30]
\draw[thick,red] (-0.1,-0.1) -- (0,0) -- (0.1,-0.1);
\end{scope}
\fill[red] (7*\s+0.2,-2.2) circle [radius=0.075];
\fill[red] (11*\s-0.2,-2.2) circle [radius=0.075];
\draw[very thick, red] (6*\s,3) -- (4*\s,3) -- (4*\s,2) -- (6*\s,2);
\draw[thick,red] (5*\s+0.1,3.1) -- (5*\s,3) -- (5*\s+0.1,2.9);
\fill[red] (6*\s,3) circle [radius=0.075];
\fill[red] (6*\s,2) circle [radius=0.075];
\node at (-0.5,0) {\footnotesize $\rho\from \tast$};
\node[above] at (2*\s,0.2) {\footnotesize $g_1\rho$};
\node[right] at (6*\s,3) {\footnotesize $g_2\rho$};
\node[right] at (11*\s,-2*\s) {\footnotesize $g_3\rho$};
\node at (12*\s+1.4,0) {\footnotesize $x_1^2 x_2x_3x_2\inv x_3\inv\tast$};
\node at (\s,-0.4) {\footnotesize $a$};
\node at (3*\s,-0.4) {\footnotesize $x_1a$};
\node at (5*\s,-0.4) {\footnotesize $x_1^2b$};
\node at (7*\s,-0.4) {\footnotesize $x_1^2 x_2c$};
\node at (9*\s,-0.6) {\footnotesize $x_1^2 x_2 x_3x_2\inv b$};
\node at (11*\s+0.3,-0.4) {\footnotesize $x_1^2 x_2 x_3x_2\inv x_3\inv c$};
\node[right] at (4*\s,2.5) {\footnotesize $x_1^3x_2x_3x_2\inv b$};
\node at (0.3,1.1) {\footnotesize $b$};
\node at (0.3,-1.1) {\footnotesize $c$};
\node[blue] at (3*\s,0.6) {\footnotesize $q_0 + q_1$};
\node[blue,left] at (4*\s-0.2,2.5) {\footnotesize $-q_1 + q_2$};
\node[blue] at (9*\s,0.4) {\footnotesize $-q_0 + q_3$};
\end{tikzpicture}
\caption{The Nielsen 1--chain in Example~\ref{ex:realization} and the translates determining the 1--chain $x$.  The coefficient of $x$ on the overlapping edges is also given.}\label{fig:realization}
\end{figure}
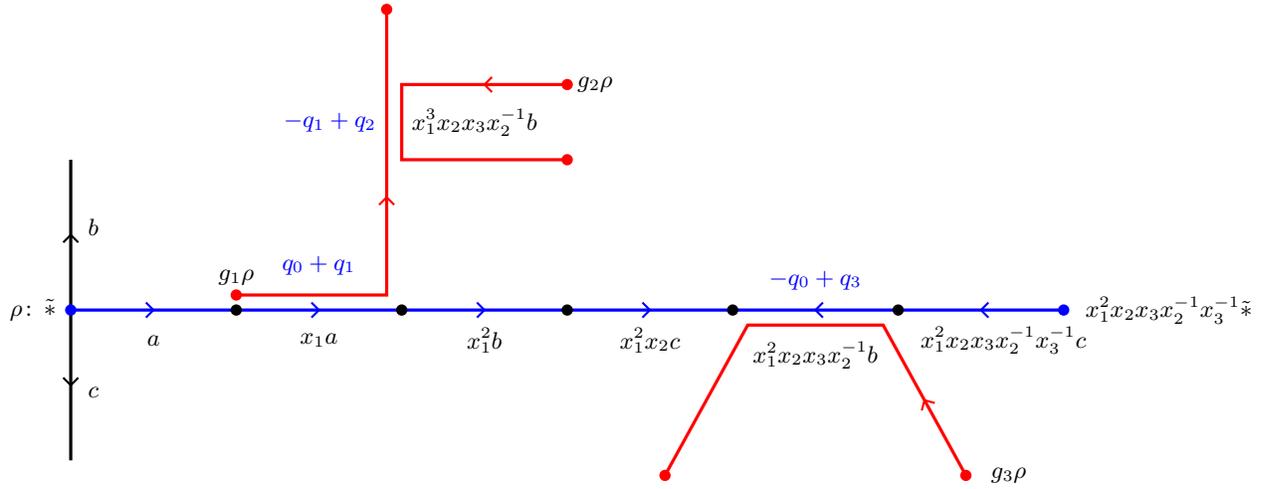

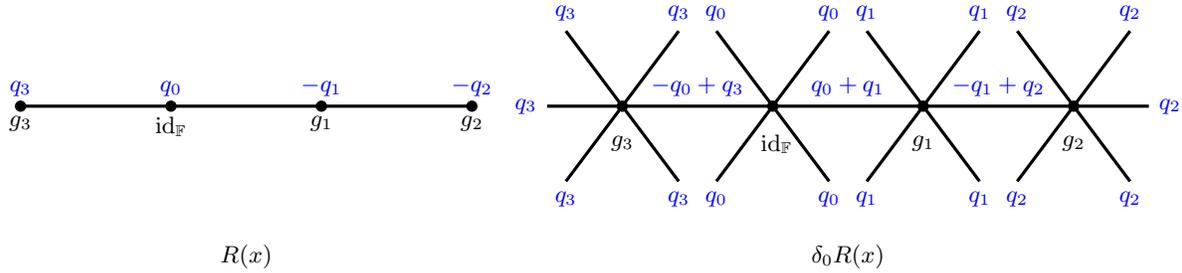
\begin{figure}[ht]
\centering
\begin{tikzpicture}
\begin{scope}
\draw[very thick] (0,0) -- (6,0);
\fill (0,0) circle [radius=0.075];
\fill (2,0) circle [radius=0.075];
\fill (4,0) circle [radius=0.075];
\fill (6,0) circle [radius=0.075];
\node[below] at (0,0) {\footnotesize $g_3$};
\node[below] at (2,0) {\footnotesize $\id_\FF$};
\node[below] at (4,0) {\footnotesize $g_1$};
\node[below] at (6,0) {\footnotesize $g_2$};
\node[above,blue] at (0,0) {\footnotesize $q_3$};
\node[above,blue] at (2,0) {\footnotesize $q_0$};
\node[above,blue] at (4,0) {\footnotesize $-q_1$};
\node[above,blue] at (6,0) {\footnotesize $-q_2$};
\node at (3,-2) {\footnotesize $R(x)$};
\end{scope}
\begin{scope}[xshift=8cm]
\draw[very thick] (-1,0) -- (7,0);
\draw[very thick] (1.25-2,1) -- (0,0);
\draw[very thick] (2.75-2,1) -- (0,0);
\draw[very thick] (1.25-2,-1) -- (0,0);
\draw[very thick] (2.75-2,-1) -- (0,0);
\draw[very thick] (1.25,1) -- (2,0);
\draw[very thick] (2.75,1) -- (2,0);
\draw[very thick] (1.25,-1) -- (2,0);
\draw[very thick] (2.75,-1) -- (2,0);
\draw[very thick] (1.25+2,1) -- (4,0);
\draw[very thick] (2.75+2,1) -- (4,0);
\draw[very thick] (1.25+2,-1) -- (4,0);
\draw[very thick] (2.75+2,-1) -- (4,0);
\draw[very thick] (1.25+4,1) -- (6,0);
\draw[very thick] (2.75+4,1) -- (6,0);
\draw[very thick] (1.25+4,-1) -- (6,0);
\draw[very thick] (2.75+4,-1) -- (6,0);
\fill (0,0) circle [radius=0.075];
\fill (2,0) circle [radius=0.075];
\fill (4,0) circle [radius=0.075];
\fill (6,0) circle [radius=0.075];
\node at (0,-0.5) {\footnotesize $g_3$};
\node at (2.05,-0.5) {\footnotesize $\id_\FF$};
\node at (4,-0.5) {\footnotesize $g_1$};
\node at (6,-0.5) {\footnotesize $g_2$};
\node at (3,-2) {\footnotesize $\delta_0R(x)$};
\node[above,blue] at (1,0) {\footnotesize $-q_0 + q_3$}; 
\node[above,blue] at (3,0) {\footnotesize $q_0 + q_1$}; 
\node[above,blue] at (5,0) {\footnotesize $-q_1 + q_2$};
\node[above,blue] at (-0.75,1) {\footnotesize $q_3$}; 
\node[above,blue] at (0.75,1) {\footnotesize $q_3$}; 
\node[below,blue] at (-0.75,-1) {\footnotesize $q_3$}; 
\node[below,blue] at (0.75,-1) {\footnotesize $q_3$}; 
\node[left,blue] at (-1,0) {\footnotesize $q_3$}; 
\node[above,blue] at (-0.75+2,1) {\footnotesize $q_0$}; 
\node[above,blue] at (0.75+2,1) {\footnotesize $q_0$}; 
\node[below,blue] at (-0.75+2,-1) {\footnotesize $q_0$}; 
\node[below,blue] at (0.75+2,-1) {\footnotesize $q_0$}; 
\node[above,blue] at (-0.75+4,1) {\footnotesize $q_1$}; 
\node[above,blue] at (0.75+4,1) {\footnotesize $q_1$}; 
\node[below,blue] at (-0.75+4,-1) {\footnotesize $q_1$}; 
\node[below,blue] at (0.75+4,-1) {\footnotesize $q_1$}; 
\node[above,blue] at (-0.75+6,1) {\footnotesize $q_2$}; 
\node[above,blue] at (0.75+6,1) {\footnotesize $q_2$}; 
\node[below,blue] at (-0.75+6,-1) {\footnotesize $q_2$}; 
\node[below,blue] at (0.75+6,-1) {\footnotesize $q_2$}; 
\node[right,blue] at (7,0) {\footnotesize $q_2$};
\end{scope}
\end{tikzpicture}
\caption{The 0--cochain $R(x)$ in Example~\ref{ex:realization} and its coboundary $\delta_0R(x)$.  For $\delta_0R(x)$ the values are only defined up to sign as orientations on the edges are not specified.}\label{fig:cochain}
\end{figure}
\end{example}

We can now show that $R$ has the properties mentioned previously in this section.

\begin{lemma}\label{lem:realize qausi-fix}
The $\QQ[\FF]$--module homomorphism $R \from V_{\rm qf} \to C_c^0(T_\rho;\QQ)$ satisfies the following statements for all $x \in V_{\rm qf}$.
\begin{enumerate}
\item\label{realize 1} $\norm{R(x)} = \norm{R(A_{f,H}^k(x))}$ for all $k \geq 0$, and
\item\label{realize 2} $\norm{x} = \norm{\delta_0R(x)}$.
\end{enumerate}
\end{lemma}

\begin{proof}
To simplify notation, we denote $A_{f,H}$ by $A$ in the proof.

We first verify~\eqref{realize 1}.  Let $x = q_1g_1\rho + \cdots + q_r g_r \rho \in V_{\rm qf}$.  As $A(\rho) = \rho$, for $k \geq 0$ we have that $A^k(x) = q_1\Phi^k_f(g_1)\rho + \cdots + q_r\Phi^k_f(g_r)\rho$.
Hence it is apparent that for all $k \geq 0$ that:
\begin{equation*}
\norm{R(A^k(x))}^2 = \sum_{j = 1}^r q_j^2.
\end{equation*}
This shows~\eqref{realize 1}.

Next, we verify~\eqref{realize 2}.  Let $x = q_1g_1\rho + \cdots + q_r g_r \rho \in V_{\rm qf}$.  As $\supp(g\rho) \subset \sfE(\tGamma) - \sfE(\tH)$ for all $g \in \FF$, we have:
\begin{align*}
\norm{x}^2 & = \sum_{e \in \sfE(\tGamma)} x_e^2 = \sum_{e \in \sfE(\tGamma) - \sfE(\tH)} \left( \sum_{j = 1}^r q_j(g_j\rho)_e  \right)^2.
\end{align*}
For the 1--cochain $\delta_0R(x) \in C_c^1(T_\rho;\QQ)$, we have:
\begin{equation*}
\norm{\delta_0R(x)}^2 = \sum_{\varepsilon \in \sfE(T_\rho)} \bigl(R(x)(\sft(\varepsilon)) - R(x)(\sfo(\varepsilon))\bigr)^2.
\end{equation*}  
As explained at the beginning of this section, an edge $e \in \sfE(\tGamma) - \sfE(\tH)$ corresponds bijectively to the edge $\varepsilon \in \sfE(T_\rho)$ where $o(\varepsilon) = g$ and $t(\varepsilon) = g'$ and $g,g' \in \FF$ are the unique elements so that $e$ is the unique edge in the intersection $\supp(g\rho) \cap \supp(g'\rho)$.  We will prove item~\eqref{realize 2} by showing that:
\begin{equation}\label{eq:equal}
\left(\sum_{j = 1}^r q_j(g_j\rho)_e\right)^2 = \bigl(R(x)(g') - R(x)(g)\bigr)^2.
\end{equation}
There are three cases depending on the cardinality of $\{g,g'\} \cap \{g_1,\ldots,g_r\}$.  

If $\{g,g'\} \cap \{g_1,\ldots, g_r\} = \emptyset$, then $(g_j\rho)_e = 0$ for all $1 \leq j \leq r$ and $R(x)(g') = R(x)(g) = 0$ and thus both sides of \eqref{eq:equal} are equal to 0.

Next, suppose that $\{g,g'\} \cap \{g_1,\ldots, g_r\} = \{g_{j_1}\}$.  Then $(g_j\rho)_e \neq 0$ only if $j = j_1$ and so the left-hand side of~\eqref{eq:equal} is equal to $q_{j_1}^2$.  Likewise, assuming without loss of generality that $g' = g_{j_1}$, we have $R(x)(g') = \sign(g')q_{j_1}$ and $R(x)(g) = 0$ and so the righthand side of~\eqref{eq:equal} is also equal to $q_{j_1}^2$.

Lastly, suppose that $\{g,g'\} \cap \{g_1,\ldots, g_r\} = \{g_{j_1},g_{j_2}\}$.  The left-hand side of \eqref{eq:equal} is equal to $(q_{j_1}(g_{j_1}\rho)_e + q_{j_2}(g_{j_2}\rho)_e)^2$.  If $(g_{j_1}\rho)_e = (g_{j_2}\rho)_e$, i.e., $e$ is non-orientable, then this quantity equals $(q_{j_1} + q_{j_2})^2$.  Else if $(g_{j_1}\rho)_e = -(g_{j_2}\rho)_e$, then this quantity equals $(q_{j_1} - q_{j_2})^2$.  Without loss of generality, we assume $g' = g_{j_1}$ and $g = g_{j_2}$ so that the righthand side of \eqref{eq:equal} equals $(\sign(g_{j_1})q_{j_1} - \sign(g_{j_2})q_{j_2})^2$.  If $(g_{j_1}\rho)_e = (g_{j_2}\rho)_e$, then $\sign(g_{j_2}) = - \sign(g_{j_1})$ and so in this case, the righthand side equals $(q_{j_1} + q_{j_2})^2$.  Else if $(g_{j_1}\rho)_e = -(g_{j_2}\rho)_e$, then $\sign(g_{j_2}) = \sign(g_{j_1})$ and so in this case the righthand side equals $(q_{j_1} - q_{j_2})^2$.

This completes the verification of~\eqref{eq:equal} and thus completes the proof of item~\eqref{realize 2}.
\end{proof}

We can now establish that the coboundary operator is bi-Lipschitz.

\begin{lemma}\label{lem:coboundary estimate}
There is a constant $B \geq 1$ such that for any $\psi \in C_c^0(T_\rho;\QQ)$:
\begin{equation*}
B\inv\norm{\psi} \leq \norm{\delta_0\psi} \leq B\norm{\psi}.
\end{equation*}
\end{lemma}

\begin{proof}
Boundedness of the coboundary operator $\delta_0 \from C_c^0(T_\rho;\QQ) \to C_c^1(T_\rho;\QQ)$ is well-known and a simple calculation that we reproduce here.  We note that any vertex in $T_\rho$ has degree $d = \#\abs{\supp(\rho)}$.  We compute:
\begin{align*}
\norm{\delta_0\psi}^2 &= \sum_{\varepsilon \in \sfE(T_\rho)} \delta_0\psi(\varepsilon)^2 = \sum_{\varepsilon \in \sfE(T_\rho)} \bigl(\psi(\sft(\varepsilon)) - \psi(\sfo(\varepsilon))\bigr)^2 \\
& \leq \sum_{\varepsilon \in \sfE(T_\rho)} 4\max\{\psi(\sft(\varepsilon))^2,\psi(\sfo(\varepsilon))^2\} \\
& \leq 4d\sum_{g \in \sfV(T_\rho)} \psi(g)^2 = 4d\norm{\psi}^2.
\end{align*}
The last inequality is observed by organizing edges based on which vertex provides the maximal value.

For the other direction, we fix a component $T_0 \subseteq T_\rho$ and let $\FF_0 = \stab(T_0) \subseteq \FF$.  As any two components of $T_\rho$ are isomorphic are graphs, it suffices to prove lower bound for $\psi \in C_c^0(T_0,\QQ)$.  Without loss of generality, we assume that $T_0$ contains the vertex $\id_\FF$.  As $\FF_0$ acts freely and transitively on the vertices of $T_0$, this implies that $\FF_0$ is finitely generated.  By~\ref{gnc:noncommuting}, $\FF_0$ is nonabelian.    

We consider the $\QQ[\FF_0]$--modules $C_c^0(T_0,\QQ)$ and $C_c^1(T_0;\QQ)$ respectively as submodules of the Hilbert spaces of square summable functions $\psi \from \sfV(T_0) \to \CC$, denoted $L^2(\sfV(T_0))$ and square summable functions $\psi \from \sfE(T_0) \to \CC$, denoted $L^2(\sfE(T_0))$, respectively.  These are finite dimensional Hilbert--$\FF_0$--modules isomorphic to $L^2(\FF_0)$ and $L^2(\FF_0)^n$ where $n$ is the number of $\FF_0$--orbits of edges in $T_0$.  The calculation above shows that the coboundary operator extends to a bounded operator $\delta_0\from L^2(\sfV(T_0)) \to L^2(\sfE(T_0))$, which is clearly $\FF_0$--equivariant.  It is easy to see that $\delta_0$ is injective as the only functions with coboundary equal to the zero function are the constant functions and the only constant function which is square summable is the zero function.  We will prove the lower bound for this operator.  This will involve several definitions and some notation that is not need elsewhere and so we refer the reader to the book by L\"uck~\cite{bk:Luck02} for these details as cited below.    

As $\FF_0$ is nonabelian and acts free and cocompactly on $T_0$, the first Novikov--Shubin invariant $\alpha_1(T_0)$ is equal to $+\infty$~\cite[Thoerem~2.55(5b)]{bk:Luck02}.  By \cite[Lemma~2.3~\&~Lemma~2.4]{bk:Luck02}, this implies that there is a constant $\lambda > 0$ such that:
\begin{equation*}
\dim_{\FF_0} \left(\img E^{\bd_1\delta_0}_{\lambda^2}\right) = F_{\delta_0}(\lambda) = \dim_{\FF_0} (\ker \delta_0) = 0.
\end{equation*}
(Here $\bd_1 \from L^2(\sfE(T_0))) \to L^2(\sfV(T_0))$ is the usual boundary operator---which is the adjoint of $\delta_0$,$\{E^{\bd_1\delta_0}_{\lambda}\}$ is the spectral family of $\bd_1\delta_0$ and $F_{\delta_0}\from [0,\infty) \to [0,\infty)$ is the spectral density function of $\delta_0$.)  Hence, $E^{\bd_1\delta_0}_{\lambda^2}(\psi) = 0$ for all $\psi \in L^2(\sfV(T_0))$ and by \cite[Lemma~2.2(2)]{bk:Luck02}, this gives that $\norm{\delta_0\psi} > \lambda\norm{\psi}$ for all $\psi \in L^2(\sfV(T_0))$ which are non-zero.  This is the desired lower bound.

Thus setting $B = \max\{2\sqrt{d},\lambda\inv\}$ completes the proof of the lemma.
\end{proof}

With these estimates, the proof of Theorem~\ref{th:quasi fixed dynamics} the geometric case is similar to the proof in the non-geometric case.

\begin{proposition}\label{prop:quasi fixed dynamics geom}
If $V_{\rm qf}$ is generated by a geometric Nielsen 1--chain, then there is a constant $C > 0$ such that for any $\xi \in V_{\rm qf}^{(2)}$ and $k \geq 0$ we have $C\inv\norm{\xi} \leq \norm{A_{f,H}^{k}(\xi)} \leq C\norm{\xi}$.
\end{proposition}

\begin{proof}
Let $B \geq 1$ be the constant from Lemma~\ref{lem:coboundary estimate} and set $C = B^2$.  To simplify notation, we denote $A_{f,H}$ by $A$ in the proof.

We first prove the proposition for $x \in V_{\rm qf}$.  By Lemmas~\ref{lem:realize qausi-fix} and \ref{lem:coboundary estimate} we find for any $k \geq 0$:
\begin{equation*}
\norm{A^k(x)} = \norm{\delta_0RA^k(x)} \leq B\norm{R(A^k(x))} = B\norm{R(x)} \leq B^2\norm{\delta_0R(x)} = C\norm{x}.
\end{equation*}    
Similarly for any $k \geq 0$:
\begin{equation*}
\norm{x} = \norm{\delta_0R(x)} \leq B\norm{R(x)} = B\norm{R(A^k(x))} \leq B^2\norm{\delta_0R(A^k(x))} = C\norm{A^k(x)}.
\end{equation*}
This proves the proposition for $x \in V_{\rm qf}$.

The general case $\xi \in V_{\rm qf}^{(2)}$ now proceeds exactly as in Proposition~\ref{prop:quasi fixed dynamics nongeom}.
\end{proof}


\section{Isolating the Quasi-Fixed Subspace}\label{sec:isolating}

The purpose of this section is Theorem~\ref{th:intersection} which proves that $V_{\rm qf}^{(2)}$ equals the intersection $\frakE(A_{f,H},\nu) \cap \frakF(A_{f,H},\nu\inv)$ for some $\nu$ sufficiently close to and greater than 1 when $f$ satisfies the chain flare condition.  We begin by showing that we can extend the chain flaring behavior from rational 1--chains to $L^2$--1--chains.

For $0 < \theta < 1$ we set:
\begin{multline*}
N_\theta(V_{\rm h}^{(2)}) = \bigl\{ \xi \in L^2(\FF)^{n_H} \mid \I{\Re(\xi),x} > \theta \norm{\Re(\xi)}\norm{x} \mbox{ and } \\  \I{\Im(\xi),x'} > \theta \norm{\Im(\xi)}\norm{x'} \mbox{ for some } x, x' \in V_{\rm h} \bigr\}
\end{multline*}
where $\Re(\param)$ and $\Im(\param)$ denote the real part and imaginary part respectively.

\begin{proposition}\label{prop:complex growth}
Suppose that the homotopy equivalence $f \from \Gamma \to \Gamma$ satisfies the chain flare condition relative to the $f$--invariant subgraph $H \subset \Gamma$.  Then there exist constants $\lambda > 1$, $0 < \theta < 1$ and $N > 0$ such that for any $\xi \in N_\theta(V_{\rm h}^{(2)})^\infty$ we have:
\begin{equation*}
\lambda\norm{A_{f,H}^{N}(\xi)} \leq \max\left\{\norm{A_{f,H}^{2N}(\xi)},\norm{\xi}\right\}.
\end{equation*}
\end{proposition}

\begin{proof}
Let $\lambda_0 > 1$ and $0 < \theta_0 < 1$ be the constants from \ref{cfh:2} for rational 1--chains in $N_{\theta_0}(V_{\rm h})^\infty$.  Let $N \in \NN$ be such that $\lambda_0^N > 2$ and set $\lambda = \lambda_0^N/2$.  Set $\theta = \frac{1 + \theta_0}{2}$.  To simplify notation, we denote $A_{f,H}$ by $A$ in the proof.
  
By Lemma~\ref{lem:induction}~\eqref{induction power}, we have that for any rational 1--chain $x \in N_{\theta_0}(V_{\rm h})^{\infty}$:
\begin{equation*}
2\lambda\norm{A^N(x)} = \lambda_0^N\norm{A^N(x)} \leq \max\left\{\norm{A^{2N}(x)}, \norm{x}\right\}.
\end{equation*}
  
Now fix a chain $\xi \in N_\theta(V_{\rm h}^{(2)})^{\infty}$ and decompose it as $\xi = \xi_{1} + i\xi_{2}$ where $\xi_1 = \Re(\xi)$ and $\xi_2 = \Im(\xi)$.  As $A$ is a real operator, we have $\norm{A^N(\xi)}^2 = \norm{A^N(\xi_1)}^2 + \norm{A^N(\xi_2)}^2$.  Without loss of generality, we may assume that $\norm{A^N(\xi_{1})}^2 \geq \frac{1}{2}\norm{A^N(\xi)}^2$.  

Let $\epsilon > 0$.  There is a rational 1--chain $x \in N_{\theta_0}(V_{\rm h})^{\infty}$ such that each of:
\[2\lambda\norm{A^{N}(\xi_{1})-A^{N}(x)}, \norm{A^{2N}(\xi_{1}) - A^{2N}(x)}, \mbox{ and } \norm{\xi_{1} - x} \] is less than $\epsilon$.
Using the observation above, we find:
\begin{align*}
2\lambda\norm{A^{N}(\xi_{1})} &\leq 2\lambda\norm{A^{N}(x)} + \epsilon \\
&\leq \max\left\{\norm{A^{2N}(x)}, \norm{x}\right\} + \epsilon  \\
&\leq \max\left\{\norm{A^{2N}(\xi_1)}, \norm{\xi_1}\right\} + 2\epsilon. 
\end{align*}
As this holds for all $\epsilon > 0$, we have $2\lambda\norm{A^{N}(\xi_{1})} \leq \max\left\{\norm{A^{2N}(\xi_{1})}, \norm{\xi_1}\right\}$.  Therefore:
\begin{align*}
\lambda\norm{A^{N}(\xi)} & \leq  2\lambda\norm{A^{N}(\xi_{1})} \\
&\leq \max\left\{\norm{A^{2N}(\xi_{1})}, \norm{\xi_1}\right\} \\
&\leq \max\left\{\norm{A^{2N}(\xi)}, \norm{\xi}\right\}.
\end{align*}
The last inequality again uses the fact that $A$ is a real operator so 
that:
\begin{equation*}
\norm{A^{2N}(\xi_{1})}^{2} \leq \norm{A^{2N}(\xi_{1})}^{2} + \norm{A^{2N}(\xi_{2})}^{2} = \norm{A^{2N}(\xi)}^{2}.\qedhere
\end{equation*}
\end{proof}

The proof of Lemma~\ref{lem:induction} carries over to $L^2$--1--chains $\xi \in N_\theta(V_{\rm h}^{(2)})^{\infty}$ using Proposition~\ref{prop:complex growth} in place of \ref{cfh:2}.

\begin{lemma}\label{lem:induction complex}
Suppose that the homotopy equivalence $f \from \Gamma \to \Gamma$ satisfies the chain flare condition relative to the $f$--invariant graph $H \subset \Gamma$ and let $\lambda > 1$, $0 < \theta < 1$ and $N > 0$ be the constants from Proposition~\ref{prop:complex growth}.  The following statements hold.

\begin{enumerate}
\item\label{induction down complex} If $\xi \in N_\theta(V_{\rm h}^{(2)})^{\infty}$, $j \geq 1$ and $\lambda\norm{A^{jN}_{f,H}(\xi)} \leq \norm{A_{f,H}^{(j-1)N}(\xi)}$, then $\lambda^j\norm{A^{jN}_{f,H}(\xi)} \leq \norm{\xi}$.

\item\label{induction up complex} If $\xi \in N_\theta(V_{\rm h}^{(2)})^{\infty}$, $j \geq 1$ and $\lambda\norm{\xi} \leq \norm{A^N_{f,H}(\xi)}$, then $\lambda^j\norm{\xi} \leq \norm{A_{f,H}^{jN}(\xi)}$.
\end{enumerate}
\end{lemma}

We can now prove the main result of this section.

\begin{theorem}\label{th:intersection}
Suppose that the homotopy equivalence $f \from \Gamma \to \Gamma$ satisfies the chain flare condition relative to the $f$--invariant graph $H \subset \Gamma$. Then there is a constant $\lambda > 1$ such that for any $1 \leq \nu < \lambda$ we have 
\begin{equation*}
V^{(2)}_{\rm qf} = \frakE(A_{f,H},\nu) \cap \frakF(A_{f,H},\nu\inv).
\end{equation*}
\end{theorem}

\begin{proof}
Let $\lambda_0 > 1$, $0 < \theta < 1$ and $N > 0$ be the constants from Proposition~\ref{prop:complex growth} and set $\lambda = \lambda_0^{1/N}$.  Fix a number $1 \leq \nu < \lambda$.  To simplify notation we denote $A_{f,H}$ by $A$ in the proof.

We begin by showing that $V^{(2)}_{\rm qf} \subseteq \frakE(A_{f,H},\nu) \cap \frakF(A_{f,H},\nu\inv)$.  Suppose that $\xi \in V^{(2)}_{\rm qf}$.  By Theorem~\ref{th:quasi fixed dynamics}, there is a $C > 0$ such that $C\inv \norm{\xi} \leq \norm{A^j(\xi)} \leq  C\norm{\xi}$ for all $j \geq 0$.  Thus $\xi \in \frakE(A,\nu)$ for any $\nu \geq 1$ as witnessed by the constant sequence $\xi_j = \xi$.  Next, take a sequence $x_j \in \CC \otimes V_{\rm qf}$ so that $\lim_{j \to \infty} \norm{x_j - \xi} = 0$.  For each $j$, there is a complex 1--chain $y_j \in \CC \otimes V_{\rm qf}$ such that $A^j(y_j) = x_j$.  Indeed, writing $x_j = z_{1,j}g_{1,j}\rho + \cdots + z_{r_j,j}g_{r_j,j}\rho$, we observe that $y_j = z_{1,j}\Phi_f^{-j}(g_{1,j})\rho + \cdots + z_{r_j,j}\Phi_f^{-j}(g_{r_j,j})\rho$ satisfies $A^j(y_j) = x_j$.  Then $\lim_{j \to \infty}\norm{A^j(y_j) - \xi} = 0$ and $\norm{y_j} \leq C\norm{A^j(y_j)} \leq 2C\norm{\xi}$ for large enough $j$ showing that $\xi \in \frakF(A,\nu\inv)$ for any $\nu \geq 1$ as witnessed by the sequence $\xi_j = y_j$.  This shows that $V^{(2)}_{\rm qf} \subseteq \frakE(A,\nu) \cap \frakF(A,\nu\inv)$. 

We next demonstrate that $\frakE(A_{f,H},\nu) \cap \frakF(A_{f,H},\nu\inv) \cap V^{(2)}_{\rm h} = \{0\}$.  From this it follows that $V^{(2)}_{\rm qf} = \frakE(A,\nu) \cap \frakF(A,\nu\inv)$ as claimed since $L^2(\FF)^{n_H} = V_{\rm h}^{(2)} + V_{\rm qf}^{(2)}$.

To this end, we suppose that $\xi \in \frakF(A,\nu\inv) \cap V^{(2)}_{\rm h}$ as witnessed by a sequence $(\xi_{j}) \subset L^2(\FF)^{n_H}$.  In other words, $\lim_{j \to \infty} \norm{A^{j}(\xi_{j}) -\xi} = 0$ and $\limsup_{j \to \infty} \norm{\xi_{j}}^{1/j} \leq \nu$.  As $\xi \in V^{(2)}_{\rm h}$, there is a $J \geq 0$ such that $A^j(\xi_j) \in N_\theta(V_{\rm h}^{(2)})$ for $j  \geq J$.  Hence $\xi_j \in N_\theta(V^{(2)}_{\rm h})^\infty$ for $j \geq J$. Let $S \subseteq \NN$ be the subset of $j \geq J/N$ where:  
\begin{equation*}
\lambda_0\norm{A^{jN}(\xi_{jN})} \leq \norm{A^{(j-1)N}(\xi_{jN})}.
\end{equation*}
By Lemma~\ref{lem:induction complex}~\eqref{induction down complex}, for $j \in S$ we have $\lambda_0^j\norm{A^{jN}(\xi_{jN})} \leq \norm{\xi_{jN}}$.  If $S$ is an infinite set, then 
\begin{equation*}
\limsup_{j \to \infty} \norm{\xi_{jN}}^{1/jN} \geq \lambda_0^{1/N} = \lambda > \nu,
\end{equation*} which contradicts the choice of the sequence $(\xi_{j})$.  Hence for large enough $j$, by Proposition~\ref{prop:complex growth} we must have:
\[ \lambda_0\norm{A^{jN}(\xi_{jN})} \leq \norm{A_f^{(j+1)N}(\xi_{jN})}. \] Taking
the limit as $j \to \infty$ we find that $\lambda_0\norm{\xi} \leq \norm{A^{N}(\xi)}$.

Next, we suppose that $\xi \in \frakE(A,\nu) \cap V^{(2)}_{\rm h}$ as witnessed by a sequence $(\xi_{j}) \subset L^2(\FF)^{n_H}$.  In other words, $\lim_{j \to \infty} \norm{\xi_{j} - \xi} = 0$ and $\limsup_{j \to \infty} \norm{A^{j}(\xi_{j})}^{1/j} \leq \nu$.  As $A^N$ is a bounded operator, we also have that $\lim_{j \to \infty} \norm{A^N(\xi_j) - A^N(\xi)} = 0$ and $\limsup_{j \to \infty} \norm{A^{N+j}(\xi_j)}^{1/j} \leq \nu$.  As $\xi \in V^{(2)}_{\rm h}$, there is a $J \geq 0$ such that $\xi_j \in N_\theta(V^{(2)}_{\rm h})$ for $j \geq J$.  Hence $A^N(\xi_j) \in N_\theta(V^{(2)}_{\rm h})^\infty$ for $j \geq J$.  Let $S \subseteq \NN$ be the subset of $j \geq J/N$ where:  
\begin{equation*}
\lambda_0\norm{A^{N}(\xi_{jN})} \leq \norm{A^{2N}(\xi_{jN})}.
\end{equation*}
By Lemma~\ref{lem:induction complex}~\eqref{induction up complex}, for $j \in S$, we have $\lambda_0^{j}\norm{A^{N}(\xi_{jN})} \leq \norm{A^{(j+1)N}(\xi_{jN})}$.  If $S$ is an infinite set, then 
\begin{equation*}
\limsup_{j \to \infty} \norm{A^{N + jN}(\xi_{jN})}^{1/jN} \geq \lambda_0^{1/N} = \lambda > \nu,
\end{equation*} which contradicts the choice of the sequence $(\xi_{j})$.  Hence for large enough $j$, by Proposition~\ref{prop:complex growth} we must have:
\[ \lambda_0\norm{A^{N}(\xi_{jN})} \leq \norm{\xi_{jN}}. \] Taking
the limit as $j \to \infty$ we find that $\lambda_0\norm{A^{N}(\xi)} \leq \norm{\xi}$.

Now consider $\xi \in \frakE(A,\nu) \cap \frakF(A,\nu\inv) \cap V^{(2)}_{\rm h}$.  As $\xi \in \frakF(A,\nu\inv) \cap V^{(2)}_{\rm h}$ we have $\lambda_0\norm{\xi} \leq \norm{A^{N}(\xi)}$.  As $\xi \in \frakE(A,\nu) \cap V_{\rm h}^{(2)}$ we have $\lambda_0\norm{A^N(\xi)} \leq \norm{\xi}$.  Hence $\lambda_0^2 \norm{\xi} \leq \norm{\xi}$, which is impossible if $\xi \neq 0$ as $\lambda_0 > 1$.  
\end{proof}

\begin{remark}\label{rem:invariant}
If $C_1(\tGamma,\tH;\QQ) = V_{\rm h} \oplus V_{\rm qf}$ and the $\QQ[\FF]$--module $V_{\rm h}$ is $A_{f,H}$--invariant, then the conclusion of Theorem~\ref{th:intersection} holds under the weaker hypothesis of \ref{cfh:2} where we only insist that $x \in V_{\rm h}$.  We sketch the modifications necessary for the proof of Theorem~\ref{th:intersection}.  Clearly, we still have $V^{(2)}_{\rm qf} \subseteq \frakE(A_{f,H},\nu) \cap \frakF(A_{f,H},\nu\inv)$.  Take $\xi \in \frakF(A_{f,H},\nu) \cap V^{(2)}_{\rm h}$ as witnessed by a sequence $(\xi_{j}) \subset L^2(\FF)^{n_H}$.  Write $\xi_j = \xi_j^{\rm h} + \xi_j^{\rm qf}$ where $\xi_j^{\rm h} \in V^{(2)}_{\rm h}$ and $\xi_j^{\rm qf} \in V^{(2)}_{\rm qf}$.  As $A_{f,H}^j(\xi_j^{\rm h}) \in V^{(2)}_{\rm h}$ and $A^j_{f,H}(\xi_j^{\rm qf}) \in V_{\rm qf}^{(2)}$, we have $A^j_{f,H}(\xi_j^{\rm h}) \to \xi$ and $A^j_{f,H}(\xi_j^{\rm qf}) \to 0$.  Since $\norm{\xi_j^{\rm qf}} \leq C\norm{A^j_{f,H}(\xi_j^{\rm qf})}$, we have $\xi_j^{\rm qf} \to 0$ and hence: \begin{equation*}
\limsup_{j \to \infty} \norm{A^j_{f,H}(\xi_j^{\rm h})}^{1/j} = \limsup_{j \to \infty} \norm{A^j_{f,H}(\xi_j)}^{1/j}.
\end{equation*}
In other words, we can assume that the sequence witnessing $\xi \in \frakF(A_{f,H},\nu) \cap V^{(2)}_{\rm h}$ lies in $V_{\rm h}^{(2)}$.  Thus as long as \ref{cfh:2} holds for elements of $V^{(2)}_{\rm h}$ we can still conclude that $\lambda_0\norm{\xi} \leq \norm{A^N(\xi)}$.  

A similar statement is true for $\xi \in \frakE(A_{f,H},\nu) \cap V^{(2)}_{\rm h}$.  Hence we may weaken the hypothesis on \ref{cfh:2} to $x \in V_{\rm h}$.  The chain flare assumption is written with the subset $N_\theta(V_{\rm h})^\infty$ as it is not obvious how to find an invariant direct sum complement to $V_{\rm qf}$.    
\end{remark}


\section{The Restriction to the Quasi-Fixed Subspace}\label{sec:restriction}

The purpose of this section is two-fold.  Firstly, we will compute the determinant of the operator $I - L^k_{f,H}$ restricted to the subspace $W^{(2)}_{\rm qf} \subset L^2(G_\phi)^{n_H}$ corresponding to the quasi-fixed subspace $V^{(2)}_{\rm qf} \subset L^2(\FF)^{n_H}$.  This takes place in Section~\ref{subsec:induced quasi-fixed} (Theorem~\ref{th:determinant on quasi-fixed}).  To make the statement precise, we first recall the notion of induction of Hilbert--$G$--modules and morphisms which takes place in Section~\ref{subsec:induction}.  Secondly, we extend Theorem~\ref{th:intersection} to the operator $L_{f,H}$ showing the equality between $W^{(2)}_{\rm qf}$ and the intersection $\frakE(L_{f,H},\nu) \cap \frakF(L_{f,H},\nu\inv)$ for $\nu$ sufficiently close to and greater than 1 when $f$ satisfies the chain flare condition.


\subsection{Induction}\label{subsec:induction}

Let $G$ be a countable group and $H$ a subgroup of $G$.  By $\iota \from H \to G$ we denote the natural inclusion.  Given a Hilbert--$H$--module $U$, the Hilbert space completion of $\CC[G] \otimes_{\CC[H]} U$ is a Hilbert--$G$--module denoted $\iota_*U$.  A morphism $A \from U \to V$ of Hilbert--$H$--modules induces a morphism $\iota_* A \from \iota_* U \to \iota_* V$ in the obvious way.  For more details see~\cite[Section~1.1.5]{bk:Luck02}.  The main properties of induction we use with regards to the Fuglede--Kadison determinant and the Brown measure are recorded below.  

\begin{lemma}\label{lem:induction properties}
Let $\iota \from H \to G$ be an injective group homomorphism and let $A \from U \to V$ be a morphism of finite dimensional Hilbert--$H$--modules.  The following statements hold.
\begin{enumerate}
\item\label{ind det} $\detG(\iota_{*}A) = \detG[H](A)$.

\item\label{ind measure} $\mu_{\iota_*A} = \mu_{A}$.

\end{enumerate}
\end{lemma}

\begin{proof}
Item~\eqref{ind det} is~\cite[Theorem~3.14~(6)]{bk:Luck02}.

Item~\eqref{ind measure} follows from \eqref{ind det} as $\mu_A$ is the Riesz measure associated to $\frac{1}{2\pi}\nabla^2 \log \detG[H] (A - zI)$ and $\mu_{\iota_{*}A}$ is the Riesz measure associated to $\frac{1}{2\pi}\nabla^2 \log \detG (\iota_{*}A - zI)$~\cite{col:Brown86}.  By \eqref{ind det} these two functions are identical.
\end{proof}


\subsection{The induced quasi-fixed subspace}\label{subsec:induced quasi-fixed}

Let $f \from \Gamma \to \Gamma$ be a homotopy equivalence and let $H \subset \Gamma$ be an $f$--invariant subgraph.  We will use the set-up and notation from Section~\ref{subsec:compute}.  In particular, we have an presentation of $G_\phi$ as a semi-direct product $G_\phi \cong \FF \rtimes_{\Phi_f} \I{t}$.  The inclusion $\iota \from \FF \to G_\phi$ gives rise to an inclusion $V^{(2)}_{\rm qf} \subseteq L^2(\FF)^{n_H} \subseteq L^2(G_\phi)^{n_H}$ and we define $W^{(2)}_{\rm qf}$ to be the closure of $\bigoplus_{\ell \in \ZZ} t^\ell V^{(2)}_{\rm qf}$ in $L^2(G_\phi)^{n_H}$.  In other words, $W^{(2)}_{\rm qf} = \iota_* V^{(2)}_{\rm qf}$.  

\begin{proposition}\label{prop:W quasi-fixed}
If $W^{(2)}_{\rm qf}$ is nontrivial, then it is isomorphic to $L^2(G_\phi)$ as a Hilbert--$G_\phi$--module.
\end{proposition}

\begin{proof}
We will show that if $V^{(2)}_{\rm qf}$ is nontrivial then it is isomorphic to $L^2(\FF)$.  This implies the statement of the proposition as $L^2(G_\phi) =\iota_* L^2(\FF)$.  

Let $\rho = \pi_H^\perp([u,v]) \in C_1(\tGamma,\tH;\QQ)$ be the Nielsen 1--chain generating $V_{\rm qf}$.  We define an $\FF$--equivariant operator $O \from \QQ[\FF] \to V_{\rm qf}$.  Given $x = q_1g_1 + \cdots + q_rg_r \in \QQ[\FF]$, we set $O(x) = q_1g_1\rho + \cdots + q_rg_r\rho \in V_{\rm qf}$.  Clearly, this map is a $\QQ[\FF]$--module surjective homomorphism.  We will show that there is a constant $D \geq 1$ such that:
\begin{equation*}
D\inv\norm{x} \leq \norm{O(x)} \leq D\norm{x}
\end{equation*}
This shows that $O$ is injective and extends to a Hilbert--$\FF$--module isomorphism $O\from L^2(\FF) \to V^{(2)}_{\rm qf}$ as claimed.  To this end, there are two cases depending on whether $\rho$ is non-geometric or geometric.

\medskip \noindent \textbf{Case 1: $\boldsymbol{\rho}$ is non-geometric.}  Let $B \geq 1$ be the constant from Lemma~\ref{lem:nongeom estimate} and set $D = \sqrt{B}$.  Given $x = q_1g_1 + \cdots + q_rg_r \in \QQ[\FF]$ we have $\norm{x}^2 = \sum_{j=1}^r q_j^2$.  Applying the estimate in Lemma~\ref{lem:nongeom estimate} with $k = 0$ we find:
\begin{equation*}
\norm{x} \leq \norm{O(x)} \leq D\norm{x}.
\end{equation*}
This proves the proposition in the the non-geometric case.

\medskip \noindent \textbf{Case 2: $\boldsymbol{\rho}$ is geometric.}  We will make use of the realization map $R \from V_{\rm qf} \to C_c^0(T_\rho;\QQ)$ and the sign map $\sign\from \FF \to \{-1,1\}$.  Let $U \from \QQ[\FF] \to C_c^0(T_\rho;\QQ)$ be the $\QQ[\FF]$--module isomorphism defined by:
\begin{equation*}
U\left(q_1g_1 + \cdots + q_rg_r\right) = q_1\sign(g_1)\chi_{g_1} + \cdots + q_r\sign(g_r)\chi_{g_r}.
\end{equation*}  (Recall $\chi_g$ is the characteristic function of the set $\{g\} \subset \sfV(T_\rho)$.)  We observe that $\norm{U(x)} = \norm{x}$ and that $U = RO$.  Let $C$ be the constant from Proposition~\ref{prop:quasi fixed dynamics geom} and set $D = \sqrt{C}$.  As shown in the proof of Proposition~\ref{prop:quasi fixed dynamics geom} with $k = 0$, we have $\norm{O(x)} \leq D\norm{RO(x)} \leq D^2\norm{O(x)}$.  As $\norm{x} = \norm{U(x)} = \norm{RO(x)}$, this proves the proposition in this case.
\end{proof}

We recall the operator $L_{f,H} \from L^2(G_\phi)^{n_H} \to L^2(G_\phi)^{n_H}$ defined in Section~\ref{subsec:compute}.  This operator is easily expressed using the current set-up as follows.  Given $\xi \in L^2(G_\phi)^{n_H}$, we write $\xi = \sum_{\ell \in \ZZ} t^\ell \xi^{(\ell)}$ where $\xi^{(\ell)} \in L^2(\FF)^{n_H}$.  Then:
\begin{equation}\label{eq:L}
L_{f,H}(\xi) = \sum_{\ell \in \ZZ} t^{\ell+1}A_{f,H}(\xi^{(\ell)}).
\end{equation}    
In particular we see that $W^{(2)}_{\rm qf}$ is $L_{f,H}$--invariant as $V^{(2)}_{\rm qf}$ is $A_{f,H}$--invariant. 

\begin{theorem}\label{th:determinant on quasi-fixed}
Suppose that the homotopy equivalence $f \from \Gamma \to \Gamma$ satisfies the chain flare condition relative to the $f$--invariant graph $H \subset \Gamma$.  Then for any $k \geq 0$, we have
\begin{equation*}
\log \detG[G_\phi] \bigl(I - L_{f,H}^k\bigr)\big|_{W^{(2)}_{\rm qf}} = 0.
\end{equation*}
\end{theorem}

\begin{proof}
If $W_{\rm qf}^{(2)} = \{0\}$ then the proposition holds as $\log \detG[G_\phi] (0) = 0$ by definition.  Thus we assume that $W^{(2)}_{\rm qf} \neq \{0\}$ and we let $O \from  L^2(G_\phi) \to W^{(2)}_{\rm qf}$ be the isomorphism from Proposition~\ref{prop:W quasi-fixed}.  

Let $P \from L^2(\I{t}) \to L^2(\I{t})$ be the morphism given by right multiplication by $1-t^k$.  We have $\log \detG[\I{t}](P) = 0$~\cite[Example~3.22]{bk:Luck02}.  Let $\iota'\from \I{t} \to G_\phi$ be the natural inclusion.  We observe that $I - L_{f,H}^k \big|_{W^{(2)}_{\rm qf}} = O (\iota'_{*}P) O\inv$.  Indeed, consider $x = t^\ell(q_1g_1\rho + \cdots + q_rg_r\rho) \in t^\ell V_{\rm qf}$.  Then using the relation $t^k\Phi^k_f(g) = gt^k$ for all $g \in \FF$ and $k \geq 0$ we find:
\begin{align*}
(I - L_{f,H}^k)(x) &= x - t^{\ell+k}A^k_{f,H}(q_1g_1\rho + \cdots + q_rg_r\rho) \\
&= x - t^{\ell+k}(q_1\Phi^k_f(g_1)\rho + \cdots + q_r\Phi^k_f(g_r)\rho) \\
&= x - t^\ell(q_1g_1\rho + \cdots + q_rg_r\rho)t^k \\
&= x(1-t^k).
\end{align*}
Hence by Lemma~\ref{lem:det isomorphism} and Lemma~\ref{lem:induction properties}~\eqref{ind det} we find:
\begin{equation*}
\log \detG[G_\phi] \bigl(I - L_{f,H}^k\bigr)\big|_{W^{(2)}_{\rm qf}} = \log \detG[G_\phi] (\iota'_{*}P) = \log \detG[\I{t}] (P) = 0.
\end{equation*}
as claimed.
\end{proof}


\subsection{Isolating the induced quasi-fixed subspace}\label{subsec:ind quasi-fixed}

Using the expression in~\eqref{eq:L} for $L_{f,H}$ we can extend Theorem~\ref{th:intersection} to the operator $L_{f,H}$.

\begin{theorem}\label{th:unit circle}
Suppose that the homotopy equivalence $f \from \Gamma \to \Gamma$ satisfies the chain flare condition relative to the $f$--invariant graph $H \subset \Gamma$. Then there is a constant $\lambda > 1$ such that for any $1 \leq \nu < \lambda$ we have 
\begin{equation*}
W^{(2)}_{\rm qf} = \frakE(L_{f,H},\nu) \cap \frakF(L_{f,H},\nu\inv).
\end{equation*}
\end{theorem}

\begin{proof}
Let $\lambda$ be the constant from Theorem~\ref{th:intersection}.  Fix a number $1 \leq \nu < \lambda$.  To simplify notation we denote $L_{f,H}$ by $L$ and $A_{f,H}$ by $A$ in the proof.

The proof that $W^{(2)}_{\rm qf} \subseteq \frakE(L,\nu) \cap \frakF(L,\nu\inv)$ is similar to the proof of $V^{(2)}_{\rm qf} \subseteq \frakE(A,\nu\inv) \cap \frakF(A,\nu)$ in Theorem~\ref{th:intersection}.  Indeed, by~\eqref{eq:L}, it is apparent that $C\inv\norm{\xi} \leq \norm{L^j(\xi)} \leq C\norm{\xi}$ for the same constant $C \geq 1$ and for all $j \geq 0$.
 
Now suppose that $\xi \in \frakF(L,\nu\inv)$ as witnessed by the sequence $(\xi_{j}) \subset L^2(G_\phi)^{n_H}$.  In other words, $\lim_{j \to \infty} \norm{L^j(\xi_j) - \xi} = 0$ and $\lim_{j \to \infty} \norm{\xi_j}^{1/j} \leq \nu$.  Using the decompositions $\xi = \sum_{\ell \in \ZZ} t^{\ell}\xi^{(\ell)}$ and $\xi_{j} = \sum_{\ell \in \ZZ} t^{\ell}\xi_{j}^{(\ell)}$ for each $j$, we have that $L^j(\xi_j) = \sum_{\ell \in \ZZ} t^{\ell+j}A^j(\xi_j^{(\ell)})$.  Hence $\lim_{j \to \infty} \norm{A^j(\xi_{j}^{(\ell)}) - \xi^{\ell}} = 0$ in $L^2(\FF)^{n_H}$ and \begin{equation*}
\limsup_{j \to \infty} \norm{\xi^{(\ell)}_{j}}^{1/j} \leq \limsup_{j \to \infty} \norm{\xi_{j}}^{1/j} \leq \nu
\end{equation*} 
so that $\xi^{(\ell)} \in \frakF(A,\nu\inv)$ for each $\ell$.

Similarly, if $\xi \in \frakE(L,\nu)$ then writing $\xi = \sum_{\ell \in \ZZ} t^{\ell}\xi^{(\ell)}$ we have that $\xi^{(\ell)} \in \frakE(A,\nu)$ as well for each $\ell \in \ZZ$.  

Hence if $\xi \in \frakE(L,\nu) \cap \frakF(L,\nu\inv)$ then writing $\xi = \sum_{\ell \in \ZZ} t^\ell\xi^{(\ell)}$, for each $\ell \in \ZZ$ we have $\xi^{(\ell)} \in V^{(2)}_{\rm qf}$ by Theorem~\ref{th:intersection} and so $\xi \in W^{(2)}_{\rm qf}$ as desired.
\end{proof}


\section{Proof of Theorem~\ref{th:mahler}}\label{sec:integral}

In this section we prove Theorem~\ref{th:mahler}, the main theorem of the article.  It is restated below for convenience.

\begin{restate}{Theorem}{th:mahler}
Suppose that $f \from \Gamma \to \Gamma$ is a homotopy equivalence that respects the reduced filtration $\emptyset = \Gamma_0 \subset \Gamma_1 \subset \cdots \subset \Gamma_S = \Gamma$ and that $f \from \Gamma \to \Gamma$ represents the outer automorphism $\phi \in \Out(\FF)$.  If the restriction of $f$ to $\hGamma_s$ satisfies the chain flare condition relative to $\hGamma_s \cap \Gamma_{s-1}$ for each $s \in \PF(f)$, then
\begin{equation}\label{eq:mahler restate}
-\rho^{(2)}(G_\phi) = \sum_{s \in \PF(f)} \int_{1 < \abs{z}} \log \abs{z} d\mu_{L_{f,s}}.
\end{equation}
Moreover, each integral in~\eqref{eq:mahler restate} is positive and hence $-\rho^{(2)}(G_\phi) > 0$.
\end{restate}

\begin{proof}
Let $f \from \Gamma \to \Gamma$ be as in the statement of the theorem.  By Theorem~\ref{th:torsion powers}, for all $k \geq 1$ we have $\rho^{(2)}(G_{\phi^{k}}) = k\rho^{(2)}(G_{\phi})$.  According to ~\cite[Remark~4.8]{ar:Clay17-2}, for all $k \geq 1$, we have $t^{k}J_{1}(f^{k})_s = (tJ_{1}(f)_s)^{k}$.  This implies that $L_{f,s}^{k} = \iota_{*}L_{f^{k},s}$ where $\iota \from G_{\phi^{k}} \to G_{\phi}$ is the natural inclusion.  According to~\cite[Theorem~4.1]{col:Brown86} $\mu_{L_{f,s}^k}$ is the push-forward of $\mu_{L_{f,s}}$ under the map $z \mapsto z^k$.  Hence using Lemma~\ref{lem:induction properties}~\eqref{ind measure} we have:
\begin{equation*}
\int_{\abs{z} > 1} \log\abs{z} \, d\mu_{L_{f^k,s}} = \int_{\abs{z} > 1} \log\abs{z} \, d\mu_{L^k_{f,s}} =\int_{\abs{z} > 1} \log\abs{z}^k \, d\mu_{L_{f,s}} = k\int_{\abs{z} > 1} \log\abs{z} \, d\mu_{L_{f,s}}. 
\end{equation*}  
In other words, both sides of~\eqref{eq:mahler restate} scale upon replacing $f$ by a power and so we are free to assume that the image of every vertex is fixed by $f$ and the set-up in Section~\ref{subsec:compute} applies. 

Again applying Theorems~\ref{th:torsion powers} and \ref{th:det-splitting} and Lemma~\ref{lem:induction properties}~\eqref{ind det}, we see that the $L^{2}$--torsion $-\rho^{(2)}(G_\phi)$ can be expressed in the following way for any $k \geq 1$:
\begin{align*}
-\rho^{(2)}(G_{\phi}) &= -\frac{1}{k}\rho^{(2)}(G_{\phi^{k}}) = \frac{1}{k} \sum_{s \in \PF(f)}\log \detG[G_{\phi^{k}}](I - L_{f^{k},s}) \\
& = \frac{1}{k} \sum_{s \in \PF(f)} \log \detG[G_{\phi}] ( I - \iota_{*}L_{f^{k},s}) = \frac{1}{k} \sum_{s \in \PF(f)} \log \detG[G_{\phi}] ( I - L^{k}_{f,s}). 
\end{align*}
Thus, by Theorem~\ref{th:brown}~\eqref{brown integral}, using the function $h(z) = z^k$, we have for each $k \geq 1$ that:
\begin{equation*}
-\rho^{(2)}(G_\phi) = \frac{1}{k} \sum_{s \in \PF(f)} \int_\CC \log \abs{1-z^k} \, d\mu_{L_{f,s}}.
\end{equation*}  

We claim that for each $s \in \PF(f)$ that:
\begin{equation}\label{eq:limit integral}
\lim_{k \to \infty} \frac{1}{k} \int_\CC \log \abs{1-z^k} \, d\mu_{L_{f,s}} = \int_{1 < \abs{z}} \log \abs{z} \, d\mu_{L_{f,s}}.
\end{equation}
To verify~\eqref{eq:limit integral}, fix an index $s \in \PF(f)$.  Let $f_s$ denote the restriction of $f$ to the $f$--invariant connected subgraph $\hGamma_s \subseteq \Gamma$.  There is a corresponding free-by-cyclic subgroup $G_{\phi,s} \subseteq G_\phi$.  Let $H$ denote the $f_s$--invariant subgraph $\hGamma_s \cap \Gamma_{s-1} \subseteq \hGamma_s$.  For the natural inclusion $\iota'\from G_{\phi,s} \to G_\phi$, as $L_{f,s} = \iota'_{*}L_{f_s,H}$, we have $\mu_{L_{f,s}} = \mu_{L_{f_s,H}}$ by Lemma~\ref{lem:induction properties}~\eqref{ind measure}.  Hence it suffices to verify~\eqref{eq:limit integral} using the measure $\mu_{L_{f_s,H}}$.  To simplify notation, denote we denote the operator $L_{f_s,H}$ by $L$ for the remainder.   

Let $\lambda > 1$ be the constant from Theorem~\ref{th:unit circle} applied to the homotopy equivalence $f_s \from \hGamma_s \to \hGamma_s$ which satisfies the chain flare condition relative to the graph $H$ by assumption.  Fix $1 < \nu < \lambda$.  We decompose the integral along the circles $\abs{z} = \nu\inv$ and $\abs{z} = \nu$ as follows: 
\begin{multline*}
\frac{1}{k} \int_{\CC} \log \abs{ 1 - z^{k}} \, d\mu_{L} =\frac{1}{k} \int_{\abs{z} < \nu\inv} \log \abs{1 - z^{k}} \, d\mu_{L} +\frac{1}{k}\int_{\nu\inv \leq \abs{z} \leq \nu} \log \abs{1 - z^{k}} \, d\mu_{L} \\
+ \frac{1}{k} \int_{\nu < \abs{z}} \log \abs{1 - z^{k}} \, d\mu_{L}.
\end{multline*}
We treat these three integral separately.

By Theorems~\ref{th:subspaces}, \ref{th:determinant on quasi-fixed} and \ref{th:unit circle}, we have:  
\begin{equation*}
\int_{\nu\inv \leq \abs{z} \leq \nu} \log \abs{1-z^k} \, d \mu_L = \log \detG[G_{\phi,s}] (I - L^k)\Big|_{W^{(2)}_{\rm qf}} = 0.
\end{equation*}  

Notice that for all $z \in \CC$ with $\abs{z} \leq \nu\inv$, we have:
\begin{equation*}
\lim_{k \to \infty} \abs{1 - z^k}^{1/k} = 1 \quad \mbox{and} \quad
\abs{\log{\abs{1-z^k}}^{1/k}} \leq \abs{\log(1-\nu\inv)}.
\end{equation*}
As constant functions are $\mu_L$--measurable since $\mu_L(\CC) < \infty$ (Theorem~\ref{th:brown}~\eqref{brown measure}), by the Lebesgue dominated convergence theorem we find:
\begin{equation*}
\frac{1}{k}\int_{\abs{z} < \nu\inv} \log \abs{1 - z^{k}} \, d\mu_L = \int_{\abs{z} < \nu\inv} \log \abs{1 - z^{k}}^{1/k} \, d\mu_{L} \to 0, \mbox{ as } k \to \infty. 
\end{equation*}

Likewise, let $r = \max\{\norm{L},\nu\} + 1$ and thus $\mu_L(\{z \in \CC \mid \abs{z} > r \}) = 0$ by Theorem~\ref{th:brown}~\eqref{brown support}.  For all $z \in \CC$ with $\nu \leq \abs{z} \leq r$, we have: 
\begin{equation*}
\lim_{k \to \infty} \abs{1 - z^k}^{1/k} = \abs{z} \quad \mbox{and} \quad
\abs{\log{\abs{1-z^k}}^{1/k}} \leq \max\{\abs{\log(\nu - 1)},\log(1 + r)\}.
\end{equation*}
Hence, by the Lebesgue dominated convergence theorem again, we find
\begin{multline*}
\frac{1}{k}\int_{\nu < \abs{z}} \log \abs{1 - z^{k}} \, d\mu_L = \int_{\nu < \abs{z} < r} \log \abs{1 - z^{k}}^{1/k} \, d\mu_{L} \\
\to \, \int_{\nu < \abs{z} < r} \log\abs{z} \, d\mu_L = \int_{\nu < \abs{z}} \log\abs{z} \, d\mu_L \mbox{ as } k \to \infty. 
\end{multline*}

Hence, combining these three calculations we have
\begin{equation*}
\lim_{k \to \infty} \frac{1}{k}\int_\CC \log\abs{1-z^k}^{1/k} \, d\mu_L = \int_{\nu < \abs{z}} \log\abs{z} \, d\mu_L
\end{equation*}
for all $1 < \nu < \lambda$ and so \eqref{eq:limit integral} holds.  Thus
\begin{align*}
-\rho^{(2)}(G_\phi) & = \lim_{k \to \infty} \sum_{s \in \PF(f)} \frac{1}{k} \int_\CC \log\abs{1-z^k}^{1/k} \, d\mu_{L_{f,s}} = \sum_{s \in \PF(f)} \int_{1 < \abs{z}} \log\abs{z} \, d\mu_{L_{f,s}}
\end{align*}
as desired verifying~\eqref{eq:mahler restate}.

It remains to show that each of the integrals in \eqref{eq:mahler restate} is positive.  As the operator $L$ is given by right-multiplication by a matrix with coefficients in $\ZZ[G_{\phi,s}]$, we have $\log \detG[G_{\phi,s}] (L) \geq 0$ according to \cite[Theorem~13.3~(2)~and~Lemma~13.11~(4)]{bk:Luck02}.   Therefore, we find
\begin{equation*}
0 \leq \log \detG[G_{\phi,s}](L) = \int_\CC \log\abs{z} \, d\mu_L = \int_{\abs{z} < 1} \log\abs{z} \, d\mu_L + \int_{1 < \abs{z}} \log\abs{z} \, d\mu_L.
\end{equation*}
As $\log\abs{z} < 0$ when $\abs{z} < 1$ and $0 < \log\abs{z}$ when $1 < \abs{z}$, if 
\begin{equation*}
\int_{1 < \abs{z}} \log\abs{z} \, d\mu_L = 0
\end{equation*}
then $\mu_L(\{z \in \CC \mid \abs{z} \neq 1\}) = 0$.  However by Theorem~\ref{th:brown}~\eqref{brown measure} we have $\mu_L(\CC) = n_s \geq 2$ and by \cite[Main~Theorem~1.1]{ar:HS09}, Proposition~\ref{prop:W quasi-fixed} and Theorem~\ref{th:unit circle} we have $\mu_L(\{z \in \CC \mid \abs{z} = 1\}) = \dim_{G_{\phi,s}}(W^{(2)}_{\rm qf}) \leq 1$.  Hence, $\mu_L(\{z \in \CC \mid \abs{z} \neq 1\}) \neq 0$ showing that the integral is indeed positive.
\end{proof}


\section{Applying the chain flare condition}\label{sec:apply}

In this final section, we include some final remarks about the chain flare condition.  First, we explain how a CT representative $f \from \Gamma \to \Gamma$ can be used to find a natural candidate for the quasi-fixed submodule $V_{\rm qf}$ associated the homotopy equivalence $f_s \from \hGamma_s \to \hGamma_s$ that satisfies \ref{cfh:3}.  We will not give the complete definition of a CT (completely split relative train-track map), but only recall the properties we require as needed.  See the works by Feighn--Handel~\cite{ar:FH11} and Handel--Mosher~\cite{ar:HM20} for full details on CTs.  Second, we present an example for consideration in which the chain flare condition simplifies.  Lastly, we include some remarks about applying the techniques of this paper to ascending HNN-extensions over free groups.  


\subsection{A candidate for the quasi-fixed submodule}\label{subsec:candidate}

First, we recall the notion of Nielsen path from which the notion of Nielsen 1--chain is modeled.  A non-trivial edge-path $\gamma$ in $\Gamma$ is a \emph{Nielsen path} if $f(\gamma)$ is homotopic rel endpoints to $\gamma$.  In particular, the endpoints of a Nielsen path $\gamma$ are fixed.   We say the Nielsen path $\rho$ is \emph{closed} if the endpoints are the same.   In a CT, the endpoints of a Nielsen path are vertices~\cite[Definition~4.7~(4)]{ar:FH11}.  

Suppose that $f \from \Gamma \to \Gamma$ is a CT with respect to the filtration $\emptyset = \Gamma_0 \subset \Gamma_1 \subset \cdots \subset \Gamma_S$.  By definition, the filtration is reduced~\cite[Definition~4.7~(3)]{ar:FH11}.  Fix an index $s \in \PF(f)$ and let $f_s$ denote the restriction of $f$ to the connected subgraph $\hGamma_s$ and let $\FF_s$ denote the subgroup (well-defined up to conjugacy) determined by $\hGamma_s$.  Let $H$ denote the $f_s$--invariant subgraph $\hGamma_s \cap \Gamma_{s-1} \subset \hGamma_s$.  Up to reversal of orientation, there is at most one Nielsen path contained in $\hGamma_s$ that is not contained in $H$~\cite[Corollary~4.19~eg-(i)]{ar:FH11}.  

If there is no such Nielsen path, we set $V_{\rm qf} = \{0\}$.  Else, let $\gamma_s$ be the Nielsen path and let $* \in \sfV(\hGamma_s)$ be one the of the endpoints of $\gamma_s$.  We fix a lift $\tast \in \sfV(\tGamma'_s)$ of $\ast \in \sfV(\hGamma_s)$ and a lift $\tf_s \from \tGamma'_s \to \tGamma'_s$ of $f_s \from \hGamma_s \to \hGamma_s$ that fixes $\tast$.  There is a lift $\tgamma_s$ of $\gamma_s$ to $\tGamma'_s$ so that one of its endpoints is $\tast$; let $v \in \sfV(\tGamma'_s)$ be the other endpoint of $\tgamma_s$.  As $\gamma_s$ is fixed up to homotopy rel endpoints by $f_s$, we have $\tf_s(v) = v$.  We claim that $\rho = \pi_H^\perp([\tast,v]) \in C_1(\tGamma'_s;\tH;\QQ)$ is a Nielsen 1--chain.  To this end, there are two cases depending on whether or not $\gamma_s$ is closed.         

If the Nielsen path $\gamma_s$ is not closed, then there is an edge $e \in \sfE(\hGamma_s) - \sfE(H)$ is is crossed exactly once by $\gamma_s$~\cite[Fact~1.42~(1)]{ar:HM20}.  The lift of this edge to $\tGamma'_s$ contained in $\tgamma_s$ satisfies item \ref{nnc:single} and so $\rho$ is a non-geometric Nielsen--1--chain in this case.  (In this case, the stratum corresponding to $\hGamma_s$ is a termed a non-geometric stratum.)  We set $V_{\rm qf}$ to the $\QQ[\FF]$--submodule generated by $\rho$.

If $\gamma_s$ the Nielsen path is closed, then every edge in $\sfE(\hGamma_s) - \sfE(H)$ is crossed exactly twice by $\gamma_s$~\cite[Fact~1.42~(2)]{ar:HM20}.  Moreover, there exists a weak geometric model for $f_s \from \hGamma_s \to \hGamma_s$~\cite[Fact~2.3]{ar:HM20}.  This includes the following data~\cite[Definition~2.1]{ar:HM20}:  
\begin{enumerate}
\item a compact connected surface $\Sigma$ with negative Euler characteristic and non-empty boundary whose components are $\bd\Sigma = \bd_0 \Sigma \cup \cdots \cup \bd_m \Sigma$;
  
\item a 2--complex $Y$ that is the quotient of attaching $\Sigma$ to $\Gamma_{s-1}$ via given homotopically nontrivial maps $\bd_j \Sigma \to \Gamma_{s-1}$, $j = 1,\ldots,m$;

\item an embedding $\hGamma_s \hookrightarrow Y$ extending the embedding $\Gamma_{s-1} \hookrightarrow Y$ where $Y - (\hGamma_s \cup \bd_0 \Sigma)$ is an open 2--disc; and 

\item the boundary component $\bd_0 \Sigma$ is homotopic in $Y$ to $\gamma_s$.

\end{enumerate}
In this case, we claim that $\rho$ is a geometric Nielsen--1--chain.  We verify the three items in turn.

Firstly, suppose $g \in \FF_s$ is nontrivial and $\supp(\rho) \cap \supp(g\rho)$ is non-empty.  Let $\beta$ be the maximal subedge-path of the edge-path $\tgamma_s$ such that $g\beta \subseteq \tgamma_s$.  Suppose there are two edges in $\beta$ that lie in $\sfE(\tGamma'_s) - \sfE(\tH)$.  Taking an innermost pair, there is a subedge-path of $\beta$ of the form $e_1 \cdot \beta' \cdot e_2$ where $e_1,e_2 \in \sfE(\tGamma'_s) - \sfE(\tH)$ and $\beta'$ is a (possibly trivial) edge-path in $\tH$.  If $\beta'$ is trivial, the vertex corresponding to $\beta'$ has valence two as $\Sigma$ is a surface.  However, in a CT, we can always arrange that the only vertices of valence two in $\hGamma_s$ either lie in $\Gamma_{s-1}$ or else are an endpoint of a Nielsen path.  If $\beta'$ is non-trivial, then its image in $\Gamma_{s-1} \subset Y$ corresponds to one of the boundary components $\bd_j \Sigma$ for some $j = 1,\ldots,m$.  In this case, we see that the endpoints of $\beta'$ have the same image in $Y$ and the link of this vertex in $\Sigma$ is disconnected.  This is a contradiction as $\Sigma$ is a surface.  Therefore we see that $\supp(\rho) \cap \supp(g\rho)$ can contain at most one edge of $\sfE(\tGamma'_s) - \sfE(\tH)$, verifying \ref{gnc:single}. 

Next, as every edge in in $\sfE(\hGamma_s) - \sfE(H)$ is crossed exactly twice by $\gamma_s$, we see that \ref{gnc:pair} holds.

Lastly, using the language of Section~\ref{subsec:geometric}, let $T_0 \subseteq T_\rho$ be the component that contains $\id_{\FF_s}$.  The stabilizer of $T_0$ is the subgroup $\pi_1(\Sigma) \subseteq \FF_s$, which is well-defined up to conjugacy.  As this subgroup acts transitively on the vertices of $T_0$, we have that the elements of $\pi_1(\Sigma)$ corresponding to the vertices adjacent to $\id_{\FF_s}$ generate $\pi_1(\Sigma)$.  As $\Sigma$ has negative Euler characteristic, $\pi_1(\Sigma)$ is non-abelian and hence at least two of these elements do not commute, verify \ref{gnc:noncommuting}. 

Hence, $\rho$ is a geometric Nielsen--1--chain in this case.  (In this case, the stratum corresponding to $\hGamma_s$ is termed a geometric stratum.)  We set $V_{\rm qf}$ to the $\QQ[\FF]$--submodule generated by $\rho$.

As every outer automorphism $\phi \in \Out(\FF)$ has power that is represented by a CT~\cite[Theorem~4.28,~Lemma~4.42]{ar:FH11}, we see that up to replacing $\phi$ by an iterate, we have a natural choice for $V_{\rm qf}$ for which \ref{cfh:3} holds. 

\begin{remark}\label{rem:iwip}
In order to show that $-\rho^{(2)}(G_\phi) > 0$ for all $\phi \in \Out(\FF)$ that are fully irreducible, it suffices to assume that $H \subseteq \Gamma$ is a single vertex (hence can be ignored) and that $V_{\rm qf} = \{0\}$.  Indeed, if a CT map for $\phi$ has a Nielsen path, then the stable tree $T_\phi$ is geometric (as an $\RR$--tree)~\cite[Theorem~3.2]{un:BF-OuterLimits}.  Hence if CT maps for both $\phi$ and $\phi\inv$ contain Nielsen paths, then both $T_\phi$ and $T_{\phi\inv}$ are geometric.  According to \cite[Corollary~9.3]{ar:Guirardel05}, this implies that $\phi$ is induced by a pseudo-Anosov homeomorphism of a surface and hence $-\rho^{(2)}(G_\phi) > 0$ as it the volume of the corresponding mapping torus.  Therefore, we may assume that a CT map for $\phi$ or $\phi\inv$ does not have any Nielsen paths and since $G_\phi = G_{\phi\inv}$, it suffices to work with the corresponding outer automorphism.       
\end{remark}


\subsection{An example to consider}\label{subsec:example}

We formulated the chain flare condition in generality so that it could apply to a general free-by-cyclic group, where we represent the monodromy $\phi$ by a CT map.  However, there are many free-by-cyclic groups for which the chain flare condition has a simpler form.  For instance, consider the following automorphism of the free group of rank 3, $\FF= \I{x_1,x_2,x_3}$, given by:
\begin{equation*}
\Phi(x_1) = x_2, \quad \Phi(x_2) = x_3, \mbox{ and } \Phi(x_3) = x_1x_2.
\end{equation*}
The obvious topological representative of $\Phi$ on the 3--rose $f \from R_3 \to R_3$ is an irreducible train-track map with no Nielsen paths (same for $\Phi\inv$) and so one would set $H = \ast$ (the unique vertex of $R_3$),  $V_{\rm qf} = \{0\}$ and $V_{\rm h} = C_1(\widetilde{R}_3;\QQ)$.  In this case, $A_f \from C_1(\widetilde{R}_3;\QQ) \to C_1(\widetilde{R}_3,\QQ)$ is a vector space isomorphism and one may verify the chain flare condition for an appropriate power by exhibiting constants $\lambda > 1$ and $N \geq 1$ such that for any integral 1--chain $x \in C_1(\widetilde{R}_3;\ZZ)$ we have:
\begin{equation*}
\lambda \norm{x} \leq \max\left\{\norm{A_f^N(x)},\norm{A_f^{-N}(x)}\right\}.
\end{equation*} 


\subsection{Ascending HNN-extensions}\label{subsec:ascending}

An ascending HNN-extension of a free group $\FF$ is the group given by a presentation:
\begin{equation*}
\FF\ast_\Psi = \I{\FF, t \mid t\inv x t = \Phi(x) \mbox{ for } x \in \FF}
\end{equation*}
where $\Psi \from \FF \to \FF$ is an injective endomorphism.  Beyond being generalizations of free-by-cyclic groups, ascending HNN-extensions arise naturally in the study of free-by-cyclic groups and the study of injective endomorphisms has seen increased interest lately~\cite{ar:DKL17,ar:Mutanguha21,thesis:Reynolds11}.  

The discussion in Section~\ref{subsec:compute}---in particular Theorem~\ref{th:det-splitting}---holds for ascending HNN-extensions.  Where the similarity breaks down and further analysis is necessary is in Section~\ref{sec:restriction}.  The key distinction between the free-by-cyclic and ascending HNN-extension cases lies in the way that $L^2(\FF)$ sits inside $L^2(\FF \rtimes_{\Phi} \I{t})$ versus the way it sits inside $L^2(\FF\ast_\Psi)$ when $\Psi$ is non-surjective.  In both cases, the relevant Hilbert space is the closure of the direct sum of a number of copies of $L^2(\FF)$, indexed by coset representative of $\FF$ in the relevant group.  In the free-by-cyclic case, we have that the operator $L_{f,H}$ respects this direct sum decomposition as written in~\eqref{eq:L}, whereas this is not true in the ascending HNN-extension case when $\Psi$ is non-surjective.  This poses a problem in the proof of Theorem~\ref{th:unit circle}.  One can create examples for $n \geq 1$ of the form $\xi = t\inv \xi^{(1)} + g t\inv \xi^{(2)} \in L^2(\FF\ast_\Psi)$ where $\xi^{(1)},\xi^{(2)} \in L^2(\FF)$ and $g \in \FF - \Psi(\FF)$ such that $\norm{A_{f,H}^k(\xi^{(1)})}, \norm{A_{f,H}^k(\xi^{(2)})} \approx 2^k$ and yet $2^{-n}\norm{L_{f,H}(\xi)} \leq \norm{\xi}$.  See the schematic in Figure~\ref{fig:schematic} contrasting the two settings.

\begin{figure}[ht]
\centering
\begin{tikzpicture}
\fill (0,1) circle [radius=0.075] node[inner sep=0pt,label=right:{$tL^2(\FF)$}] (1) {};
\fill (0,0) circle [radius=0.075] node[inner sep=0pt,label=right:{$L^2(\FF)$}] (2) {};
\fill (0,-1) circle [radius=0.075] node[inner sep=0pt,label=right:{$t\inv L^2(\FF)$}] (3) {};
\draw[very thick] (1) -- (3);
\draw[thick,->] (-1,-0.5) -- (-1,0.5);
\node at (-1.5,0) {$L_{f}$};
\node at (0,-2.5) {$L^2(\FF \rtimes_\Phi \I{t})$};
\begin{scope}[xshift=2cm]
\fill (5,1) circle [radius=0.075] node[inner sep=0pt,label=above:{$tL^2(\FF)$}] (0) {};
\fill (3.5,0) circle [radius=0.075] node[inner sep=0pt,label={[shift={(-0.5,-0.1)}]$L^2(\FF)$}] (00) {};
\fill (2.15,-1) circle [radius=0.075] node[inner sep=0pt,label=below:{$t\inv L^2(\FF)$}] (000) {};
\fill (4,-1) circle [radius=0.075] node[inner sep=0pt,label=below:{$gt\inv L^2(\FF)$}] (001) {};
\fill (6.5,0) circle [radius=0.075] node[inner sep=0pt,label={[shift={(1,-0.1)}]$tgt\inv L^2(\FF)$}] (01) {};
\fill (6,-1) circle [radius=0.075] node[inner sep=0pt,label=below:{$t
gt^{-2} L^2(\FF)$}] (010) {};
\fill (7.85,-1) circle [radius=0.075] node[inner sep=0pt,label={[shift={(0.5,-0.7)}]$t gt\inv gt\inv L^2(\FF)$}] (011) {};
\draw[very thick] (00) -- (0) -- (01) (000) -- (00) -- (001) (010) -- (01) -- (011);
\draw[thick,->] (10,-0.5) -- (10,0.5);
\node at (10.5,0) {$L_{f}$};
\node at (5,-2.5) {$L^2(\FF \ast_\Psi)$};
\end{scope}
\end{tikzpicture}
\caption{Contrasting the free-by-cyclic setting with the ascending HNN-extension setting.}\label{fig:schematic}
\end{figure}
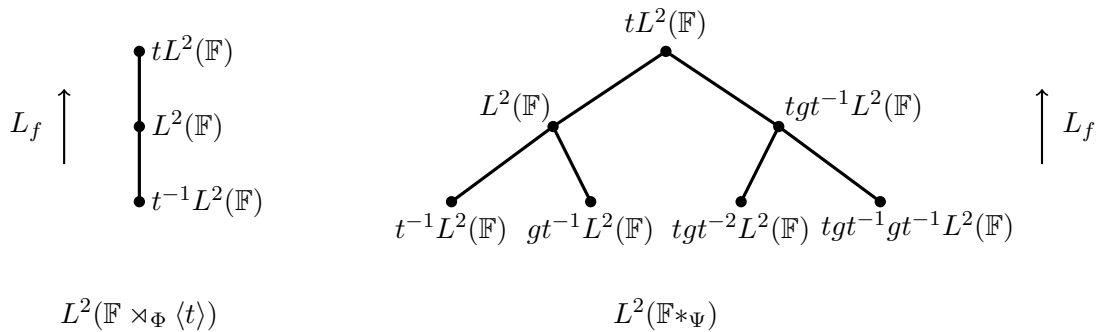


\bibliography{$HOME/Dropbox/bibliography}
\bibliographystyle{siam}

\end{document}